\documentclass[12pt,a4paper]{amsart}
\usepackage[T1]{fontenc}
\usepackage{amsmath,amsfonts,amssymb,amscd,mathtools}
\usepackage{mathrsfs}
\usepackage{float,enumitem}
\usepackage{hyperref}

\usepackage{ifthen}
\usepackage{tikz}
\usepackage{tikz-cd}
\usepackage{graphicx}

\DeclareFontFamily{U}{wncy}{}
\DeclareFontShape{U}{wncy}{m}{n}{<->wncyr10}{}
\DeclareSymbolFont{mcy}{U}{wncy}{m}{n}
\DeclareMathSymbol{\Sha}{\mathord}{mcy}{"58}

\newcommand{\tr}{\operatorname{tr}}

\newcommand{\interval}[4]{
  \ifthenelse{ \equal{#1}{o} } {\mathopen{]}} {\mathopen{[}}
  #2, #3
  \ifthenelse{ \equal{#4}{o} } {\mathclose{[}} {\mathclose{]}}
}

\renewcommand{\H}{\mathbf{H}}

\newcommand{\Z}{\mathbf{Z}}
\newcommand{\N}{\mathbf{N}}
\newcommand{\G}{\mathbf{G}}

\newcommand{\PSL}{\mathrm{PSL}}

\newcommand{\SL}{\mathrm{SL}}

\newcommand{\GL}{\mathrm{GL}}

\newcommand{\bG}{\mathbf{G}}

\newcommand{\CC}{\mathbb C}
\newcommand{\RR}{\mathbb R}
\newcommand{\ZZ}{\mathbb Z}
\newcommand{\HH}{\mathbb H}

\def\<{\langle}
\def\>{\rangle}

\def\N{\mathbb{N}}
\def\R{\mathbb{R}}
\def\C{\mathbb{C}}

\def\Z{\mathbb{Z}}
\def\HH{\mathbb{H}}
\def\H{\mathbb{H}}
\def\Q{\mathbb{Q}}
\def\S{\mathbb{S}}

\def\Prob{\text{Prob}}

\def\G{\Gamma}
\def\g{\gamma}

\DeclareMathOperator{\Ad}{Ad}

\usepackage[]{cleveref}
\newtheorem{theorem}{Theorem}[section]
\newtheorem{lemma}[theorem]{Lemma}
\newtheorem{corollary}[theorem]{Corollary}
\newtheorem{proposition}[theorem]{Proposition}

\newtheorem{question}[theorem]{Question}

\hypersetup{
  colorlinks=true,
  linkcolor=blue!55!black,
  citecolor=green!45!black,
  urlcolor=blue!65!black
}

\theoremstyle{definition}
\newtheorem{definition}[theorem]{Definition}
\newtheorem{example}[theorem]{Example}
\newtheorem{remark}[theorem]{Remark}

\numberwithin{equation}{section}

\title[Irreducible groups and ergodicity in the boundary]{Irreducible groups and ergodicity in the boundary}

\author{Subhadip Dey}
\address{Tata Institute of Fundamental Research, School of Mathematics, Homi Bhabha Road, Colaba, Mumbai 400005, India}
\email{subhadip@math.tifr.res.in}
\author{Sebastian Hurtado}
\address{Yale University, Department of Mathematics, 219 Prospect St, New Haven, CT 06511, USA}
\email{sebastian.hurtado-salazar@yale.edu}


\begin{document}
\maketitle
\begin{abstract}
 We show that if $G$ is a real semi-simple Lie group, and $\Gamma$ is a discrete subgroup of $G$ containing a subgroup $\Sigma$ acting  ergodically (in a strong sense) on the Furstenberg boundary of $G$, then $\Gamma$ is not isomorphic to a free product of $\Sigma$ with $\Z$.  Moreover, if $\Sigma$ has algebraic  entries, then $\Gamma$ has algebraic  entries as well. As a consequence, we show that if all irreducible discrete subgroups of $\SL_2(\R) \times \SL_2(\R)$ act ergodically on $\mathbb{S}^1\times \mathbb{S}^1 $, such groups cannot be free groups (or even Gromov hyperbolic). In the appendix, we discuss a connection between the existence of discrete irreducible groups and diophantine properties of Lie groups.
\end{abstract}


\begin{section}{Introduction}
Discrete subgroups of semi-simple Lie groups have been intensively studied for more than a century. However, a fairly comprehensive understanding of discrete subgroups of $\SL_2(\CC)$ (the Kleinian groups, or equivalently, the fundamental groups of complete hyperbolic three-orbifolds) has emerged only over the past three decades.
The classification of finitely generated Kleinian groups was completed through the proof of Marden’s Tameness Conjecture by Agol \cite{agol2004tameness} and independently by Calegari--Gabai \cite{calegari-gabai}, together with the proof of the Ending Lamination Conjecture by Brock, Canary, and Minsky \cite{ELC-I,ELC-II}.

Nonetheless, for general semi-simple Lie groups, even for $\SL_3(\R)$, a classification of finitely generated discrete subgroups, if such a classification is possible at all, remains far out of reach. Moreover, for higher-rank Lie groups, there is a striking scarcity of ``geometrically infinite'' examples.

For example, consider $\PSL_2(\mathbb{R}) \times \PSL_2(\mathbb{R})$. To the best of our knowledge, the only known examples of discrete subgroups are either irreducible lattices (all arithmetic by Margulis) or groups whose projection to one $\PSL_2(\mathbb{R})$ factor is discrete. The latter class includes direct products of discrete subgroups of $\PSL_2(\mathbb{R})$ (and their subgroups, e.g. fiber products), as well as subgroups that are Anosov (or, more generally, relatively Anosov, transverse, and so on).

A natural way to search for irreducible discrete subgroups of $\PSL_2(\RR) \times \PSL_2(\RR)$, i.e., discrete subgroups whose projection is  dense in each $\PSL_2(\R)$ factor, is to enlarge a diagonally embedded lattice. However, an application of our main result, \Cref{main}, shows that such constructions might only yield known classes of subgroups. This answers the first part of Question 6.12 of Benoist--Miquel \cite{Benoist-Miquel} in the case of the modular group:

\begin{theorem}\label{cor:1}
Suppose that $\Gamma<\PSL_2(\mathbb R)\times\PSL_2(\mathbb R)$ is a Zariski dense discrete subgroup whose intersection with the image of the diagonal embedding $\Delta:\PSL_2(\mathbb R)\to \PSL_2(\mathbb{R}) \times \PSL_2(\mathbb{R})$, $g \mapsto (g,g)$, is commensurable to $\Delta(\PSL_2(\Z))$.
Then $\Gamma$ is either commensurable to a subgroup of $\PSL_2(\Z) \times \PSL_2(\Z)$ or an irreducible arithmetic lattice in $\PSL_2(\mathbb R)\times\PSL_2(\mathbb R)$.
\end{theorem}

A possible explanation for the scarcity of irreducible discrete subgroups may come from their boundary dynamics.
To the best of our knowledge, the only known examples of irreducible discrete subgroups are arithmetic lattices; see the recent works \cite{brody2023greenberg,fisher2024commensurators} for related questions. Irreducible subgroups can be shown, by an elementary argument, to act minimally on the Furstenberg boundary $\S^1 \times \S^1$ and ergodically on each $\S^1$ factor. 

In this regard, it seems plausible that the following question has an affirmative answer:

\begin{question}\label{ques:irr}
    If $\Gamma < \PSL_2(\RR) \times \PSL_2(\RR)$ is irreducible, must it act ergodically on the Furstenberg boundary $\S^1\times \S^1$?
\end{question}

Furthermore, the following question, attributed to Yves Benoist (see the historical discussion in \cite{brody2023greenberg} and also \cite[Ques.\ 6.12]{Benoist-Miquel}), is open:

\begin{question}\label{ques:free}
Suppose a subgroup $\Gamma < \PSL_2(\RR) \times \PSL_2(\RR)$ is discrete and irreducible. Can $\Gamma$ be isomorphic to a free group?
\end{question}

If \Cref{ques:irr} indeed had an affirmative answer, then our main result would provide a negative answer to \Cref{ques:free}:

\begin{theorem}
    Suppose every irreducible discrete subgroup of $G = \PSL_2(\RR) \times \PSL_2(\RR)$ acts ergodically on the boundary $\S^1 \times \S^1$. Then no irreducible discrete subgroup $\Gamma$ of $G$ is hyperbolic, and, in particular, such groups $\Gamma$ cannot be free
\end{theorem}

We refer to \Cref{irreducibleergodic} for a more general statement.

\subsection{Main result}

For the rest of the article, let $G$ denote a connected noncompact real semi-simple Lie group with finite center. Let $B$ be a maximal connected amenable subgroup of $G$, and call the quotient $G/B$ the \emph{Furstenberg boundary} of $G$.

Before stating our main result, we introduce a few definitions:

\begin{definition}
We say that a subgroup $\Gamma$ of a real semi-simple Lie group $G$ has {\em algebraic traces} (respectively, {\em integral traces}) if, for any $\gamma \in \Gamma$, the trace of $\text{Ad}(\gamma)$ (the image of $\gamma$ under the adjoint representation) is an algebraic number (respectively, an algebraic integer).
\end{definition}

\begin{definition}\label{def:entries}
We say that a subgroup $\Gamma$ of a real semi-simple Lie group $G$ has {\em algebraic entries} (respectively, {\em integral entries}) if, for some faithful representation $\rho: G \to \operatorname{GL}(V)$ on a real vector space and some appropriate basis of $V$, the image $\rho(\Gamma)$ consists of matrices with entries in the algebraic numbers $\bar{\Q}$ (respectively, with entries in the algebraic integers). This property is independent of the representation; see \ref{algebraicandintegraltraces}.
\end{definition}

Finally, throughout, we will use the following nonstandard notion.

\begin{definition}
We say that a subgroup $\Sigma < G$ acts \emph{strongly ergodically} on the Furstenberg boundary if every finite-index subgroup of $\Sigma$ acts ergodically on the Furstenberg boundary of $G$.
\end{definition}

Our main result, stated below, fits into the general philosophy that sufficiently large discrete subgroups of semi-simple Lie groups have algebraic constraints (both group-theoretic and arithmetic):

\begin{theorem}[Main result]\label{main}
Let $\Gamma$ be a finitely generated Zariski-dense discrete subgroup of $G$, and suppose that $\Sigma < \Gamma$ is a subgroup such that $\Sigma$ acts strongly ergodically on the Furstenberg boundary $G/B$. Then:
\begin{enumerate}[label=(\roman*)]
\item If $\Sigma$ has algebraic entries, then $\Gamma$ has algebraic entries.
\item If $\Sigma$ has integral entries, then $\Gamma$ has integral entries.
\item For every non-trivial $\gamma \in \Gamma$, the natural homomorphism from the abstract free product $\langle \gamma \rangle * \Sigma \to \Gamma$ is not faithful.
\end{enumerate}
\end{theorem}

The main ideas in the proof of our theorems come from the work of Chatterji--Venkataramana \cite{chatterji2009discrete}, as well as from by-now-classical ideas in proofs of local rigidity and super-rigidity of lattices. 
Chatterji–Venkataramana obtain superrigidity under hypotheses involving higher-rank lattices. Our results establish group-theoretic and arithmetic restrictions under a boundary-ergodicity hypothesis, which also applies to rank-one diagonal lattices.
Nonetheless, in Appendix \ref{geomproof}, we also provide an alternative geometric proof of Item (3) in some special cases, using a more geometric argument.

Examples beyond lattices satisfying the strong-ergodicity hypothesis in \Cref{main} arise from diagonally embedded lattices $\Lambda$ in semi-simple Lie groups $H$,
$\Delta(\Lambda)<G=H\times H$. Such examples also occur when $G$ is
simple. In rank one, they include finitely generated Kleinian groups
with full limit set; see \cite{canary2008marden}. More
recently, in joint work with Frączyk \cite{DFH26}, we showed
that lattices in the standard subgroup
$\operatorname{SO}(n,1)$ of $\operatorname{SO}(n,2)$ act
strongly ergodically on the Furstenberg boundary of the latter.

\medskip

We will now discuss some consequences of our work and highlight some questions for further research.

\subsection{Subgroups of products containing diagonal lattices}

To the best of our knowledge, the following question is open:

\begin{question}
Let $G$ be a noncompact real semi-simple Lie group with trivial center, and let $\Sigma < G$ be a lattice. Let $\Delta: G \to G \times G$ be the diagonal embedding, and suppose that $\Gamma$ is an irreducible discrete subgroup of $G \times G$ that contains $\Delta(\Sigma)$. Is $\Gamma$ necessarily a lattice?
\end{question}

We remark that the condition ``irreducible'' in the above question cannot be replaced by the condition ``Zariski dense,'' as Bass--Lubotzky \cite{bass-lubotzky} constructed Zariski-dense discrete subgroups $\Gamma$ of $G\times G$, where $G = {\rm F}_4^{-20}$, that are not lattices in $G\times G$, yet whose intersection with the diagonal copy of $G$ in $G\times G$ is a lattice. Moreover, these groups are super-rigid and were the first counterexamples to a conjecture of Platonov (a conjecture that remains open for simple higher-rank Lie groups).

Venkataramana studied subgroups containing diagonally embedded lattices in various works (see, e.g., \cite{Venkataramana0,Venkataramana1}) and proved, for example, that if $G$ is a higher-rank Lie group, such subgroups must be super-rigid. He also showed that such subgroups must sometimes necessarily be lattices. (For example, he showed that if $\PSL_2(\ZZ) \subsetneq \Gamma \subset \PSL_2(\ZZ[\sqrt{2}])$, then $\Gamma$ must be a finite-index subgroup of $\PSL_2(\ZZ[\sqrt{2}])$; see \cite[Prop. 2.1]{Venkataramana1}.)

Since lattices $\Gamma$ of $G$ act doubly ergodically on the Furstenberg boundary of $G$,
we have the following corollary of Theorem \ref{main}:

\begin{corollary}\label{diagonallattice} Let $\Gamma$ be a Zariski-dense discrete subgroup of $G \times G$ such that $\Gamma$ contains $\Delta(\Sigma)$. Then $\Gamma$ is not equal to a nontrivial free product with $\Delta(\Sigma)$ as a free factor. Moreover, if $\Sigma$ has algebraic (resp. integral) entries, then $\Gamma$ has algebraic (resp. integral) entries.
\end{corollary}

\begin{remark}
Unless $\Sigma$ is a non-uniform lattice in $G = \SL_2(\R)$, it follows that $\Gamma$ is not even abstractly isomorphic to a non-trivial free product, as $\Sigma$ is one-ended and one-ended subgroups of free products are contained in one of the factors up to conjugation.
\end{remark}


Consider a closed surface of genus $g \geq 2$, and take two discrete and faithful embeddings $\rho_i: \pi_1(S) \to \PSL_2(\R)$, $i =1,2$. Let
$$\Gamma = \{ (\rho_1(\gamma), \rho_2(\gamma) ) \ \ | \ \ \gamma \in \pi_1(S) \},$$
which is a discrete subgroup of $\PSL_2(\RR)\times \PSL_2(\RR)$ (we will call these subgroups \emph{Anosov surface groups}). Based on Corollary \ref{diagonallattice}, we expect a negative answer to the following question:

\begin{question} Suppose that a Zariski-dense subgroup $\Gamma < \PSL_2(\RR) \times \PSL_2(\RR)$ contains an Anosov surface group. Can $\Gamma$ be written as a nontrivial free product?
\end{question}

It is shown in \cite{DFH26} that if $\rho_1$ not equivalent upto an automorphism of $\PSL_2(\RR)$ to $\rho_2$, then $\Gamma$ does not act ergodically on the Furstenberg boundary. 
So, Theorem \ref{main} does not give an immediate negative answer to the question above.
However, if all irreducible discrete subgroups of $\PSL_2(\RR) \times \PSL_2(\RR)$ were known to act ergodically on their Furstenberg boundary $\mathbb S^1\times \mathbb S^1$, then a negative answer would follow from Item (2) of \Cref{irreducibleergodic}.

\subsection{Hyperbolic manifolds with integral traces}

If $\Gamma$ is a lattice in $G =\text{Isom}^{+}(\H^n)$ and $n \geq 3$, then it follows from local rigidity that $\Gamma$ has algebraic entries (in any faithful representation; see \Cref{algebraictraces}). Let $\Ad(\Gamma)$ be the image of $\Gamma$ under the adjoint representation, and let $\bG$ be the Zariski closure of $\Ad(\Gamma)$.

If we let $k$ be the adjoint trace field (the field generated by traces in the adjoint representation), then it follows that, up to conjugation, $\Ad(\Gamma) < \bG(k)$ and $k$ is the smallest such field (see \cite{Vinberg1974Field}). We will identify $\Gamma$ with $\Ad(\Gamma)$. Let $O_k$ be the ring of integers in $k$.



\begin{definition} Let us define the {\em integral subgroup} of $\Gamma$ as
$$\Gamma_0 := \Gamma \cap \bG(O_k),$$
and define the \emph{integral critical exponent} $\delta_{\Gamma_0}$ as the critical exponent of $\Gamma_0$.
\end{definition}

Although the group $\Gamma_0$ is only well-defined up to finite index, the integral critical exponent is well defined. We have the following elementary proposition:

\begin{proposition}
If $\Gamma$ is a lattice in ${\rm Isom}(\HH^n)$, $n\geq2$, and $|\Gamma_0| = \infty$, then:

\begin{enumerate}
\item $\Gamma_0$ is commensurated by $\Gamma$.
\item $\Gamma_0$ has full limit set.
\item $\delta_{\Gamma_0} > 0$.
\end{enumerate}

\end{proposition}

Alan Reid has asked the following question:

\begin{question} Can $\Gamma_0 = \{e\}$?
Equivalently, can a hyperbolic manifold have non-integral traces for all non-trivial elements?
\end{question}

Not much is known about how large or small $\Gamma_0$ can be, and one can ask the following question:

\begin{question} For fixed $n \geq 2$, and among lattices $\Gamma$ in ${\rm Isom}(\HH^n)$, what are the possible values of $\delta_{\Gamma_0}$? Is this set of values dense in $[0,n-1]$?
\end{question}

Similar to the results proved by Bader--Fisher--Miller--Stover \cite{bader-miller-fisher-stover}, one can ask:

\begin{question}
Suppose $\Gamma$ is a lattice in ${\rm Isom}(\HH^3)$, and suppose that $\Gamma$ contains infinitely many surface subgroups (i.e., subgroups isomorphic to the fundamental group of a closed surface and not commensurable with one another) with integral entries. Does $\Gamma$ have integral entries? What can we say when $n > 3$?
\end{question}

The following result follows from Theorem \ref{main}:

\begin{corollary} \label{hyperbolictraces}
Suppose $\Gamma$ is a lattice in ${\rm Isom}(\H^n)$, and suppose that a subgroup $\Gamma_0 < \Gamma$ acts strongly ergodically on $\partial \H^n$ and has integral entries. Then $\Gamma$ has integral entries.
\end{corollary}

This implies, in particular, that if $\Gamma$ is a lattice in $\text{Isom}(\HH^3)$ and contains one geometrically infinite surface subgroup with integral traces, then $\Gamma$ has integral traces, as geometrically infinite surface groups always act ergodically on $\partial \HH^3$ (this last fact can be proved more elementarily using the fact that geometrically infinite surface groups virtually fiber over the circle).

\subsection{Acknowledgments}
We would like to thank many people for various conversations and discussions related to this work, among them: Yves Benoist, David Fisher, Mikolaj Fraczyk, Misha Kapovich, Moishe Kohan, Alex Lubotzky, Yair Minsky, Juan Piriz Lorenzo, Ilia Smilga, Mohan Swaminathan, and Wouter Van Limbeek.
We thank Sacha Lheveder for a correction in the proof of \Cref{free-product}.

\subsection{AI Disclosure}
The authors used LLM tools to assist with literature searches and learning purposes. All the main ideas and contributions of this manuscript are solely human made.

\end{section}

\begin{section}{Preliminaries}

Unless stated otherwise, we will use the following notation throughout this paper:

\begin{subsection}{Notation}
Let $G$ denote a semi-simple real Lie group with finite center. Let $A$ denote a choice of a  Cartan subgroup of $G$, Let $B$ be a minimal parabolic subgroup of $G$ containing $A$ ($B$ is the unique maximal connected amenable subgroup of $G$ up to conjugation) and let $K$ be a maximal compact subgroup of $G$. Let $M < K$ be the centralizer of $A$ in $K$, and $N$ be the unipotent radical of $B$. We have that $B = MAN$. We also recall the Iwasawa decomposition of $G$ given by $G = KAN$. Let us define the Furstenberg boundary to be $\mathcal{F} := G/B$.

\subsection{Strong ergodicity}

The group $G$ acts continuously on the boundary $\mathcal{F}$ by left multiplication and preserves the Lebesgue class (but does not preserve the Lebesgue measure). We will use the following definition (the second part is not standard): 

\begin{definition}\label{strongergodic}
We say that a subgroup $\Sigma < G$ acts \emph{ergodically} on the Furstenberg boundary $\mathcal{F}$ if any Lebesgue measurable set which is $\Sigma$-invariant, has zero or full Lebesgue measure. We say that $\Sigma$ acts \emph{strongly ergodically} on $\mathcal{F}$ if every finite index subgroup of $\Sigma$ acts ergodically on $\mathcal{F}$.
\end{definition}

Not every group that acts ergodically on the boundary necessarily acts strongly ergodically as it was pointed out to us on MathOverflow by Moishe Kohan \cite{MoisheKohan}. Examples of such kind are as follows:

\begin{example}\label{ex} Let $G = \PSL_2(\RR)$, consider $\Gamma_0 < G$ so that $\HH^2/\Gamma_0$ is a $\ZZ^d$-cover, $d\ge 3$, of a compact hyperbolic surface, and let $\Gamma_1 < \Gamma_0$ be the index two subgroup of $\Gamma_0$, such that $\HH^2/\Gamma_0$ is obtained from $\HH^2/\Gamma_1$ by cutting along a simple non-separating geodesic, obtaining a hyperbolic surface $S$ with two boundary geodesics, and then doubling along this geodesics. The group $\Gamma_0$ acts ergodically on $\mathbb{S}^1$ because bounded harmonic functions on $\Z^d$  are constant, but $\Gamma_1$ does not act ergodically because due to the transience of the random walk on $\ZZ^d$, a Brownian in  path in $\HH^2/\Gamma_1$ will eventually enter one of the two copies of $S$ and never come back to the other copy. See \cite{lyons-sullivan,uludag} for more details.

Consequently, if $\Gamma$ is a torsion-free uniform lattice in $\PSL_2(\R)$, then the commutator subgroup $\Gamma_0 \coloneqq [\Gamma,\Gamma]$ acts ergodically but not strongly ergodically on $\mathbb S^1$.
\end{example}

\subsection{Boundary Map}

We will make use of the following lemma, attributed to Furstenberg.  See \cite{Zimmer}, Corollary 4.3.7 and Proposition 4.3.9 for a proof. For completeness, we provide a different proof using the existence of conditional measures:

\begin{lemma}\label{map}  Suppose that $\Gamma < G$ is a discrete subgroup of a semi-simple Lie group $G$, and suppose $\Gamma$ acts continuously on a compact metric space $X$, then there exists a Borel measurable map $\phi: \mathcal{F} \to \text{Prob}(X)$ which is $\Gamma$-equivariant.
\end{lemma}

\begin{proof}

Consider the suspension space $Y := (G \times X)/\Gamma$, obtained from the action on $X$, where action of $\Gamma$ is via $\gamma . (g,x) = (g\gamma^{-1}, \gamma x) $ . $G$ acts by left multiplication on $Y$.

We can construct a Radon measure $\mu$ that projects to Haar measure in $G/\Gamma$, and which is $B$-invariant (recall that $\mathcal{F} = G/B$ and $B$ is amenable). 

Consider the space $M_Y$ of Radon measures on $Y$ that project to Haar measure on $G/\Gamma$, one can show that $M_Y$ is non-empty. Consider a Folner sequence $F_n$ of $B$, and take $\mu_0 \in M_Y$, and consider the averages:

$$\mu_n := \frac{1}{m_{B}(F_N)}\int_{F_n} p_{*} \mu_0 \ dp.$$

Take a sub-sequential limit $\mu$ of $\{\mu_n\}$. Such limit $\mu$ exists and must be $B$-invariant. We can lift $\mu$ to a measure $\hat{\mu}$ on $G \times X $ that projects to Haar measure, and which is $\Gamma$-invariant, moreover such measure is also $B$-invariant, for the right $B$-action. The existence of conditional measures with respect to the $X$-foliation implies that for Haar a.e. $g \in G$ there is a probability measure $m_g$ on $X$, and from the invariance conditions we have for Haar $a.e.$ $ g \in G$, and every $p \in B$, that $m_{pg}= m_g$, and $m_{g\gamma^{-1}} = \gamma_{*} m_{g}$. 

Therefore the assignment $g \to m_{g}$, induces a map $\phi: G/B \to \Prob(X)$ which is $\Gamma$-equivariant. 
\end{proof}

\subsection{Algebraic and integral traces}\label{algebraicandintegraltraces}

Let $G$ be a real connected semi-simple Lie group without compact factors and finite center, and let $\Gamma$ be a subgroup of $G$. Let $Ad: G \to {\rm End}(\mathfrak{g})$ be the adjoint representation. 

\begin{definition}
We say that $\Gamma$ has {\em algebraic traces} (respectively, {\em integral traces}) if for any $\gamma \in \Gamma$, the trace of $\Ad(\gamma)$ is an algebraic number (respectively, algebraic integer). We let $k_\Gamma$ denote the field generated by all such traces, and refer to $k_\Gamma$ as the \emph{adjoint trace field}.
\end{definition}

If the group $\Gamma$ is finitely generated, the adjoint trace field $k_{\Gamma}$ is finitely generated. If $\Gamma$ has algebraic traces, $[k_{\Gamma}: \Q] < \infty$.

The following is well known to experts, but we were not able to find a proof in the literature.

\begin{theorem}\label{algebraictraces} Let $G$ be a semi-simple real algebraic group with trivial center, and $\Gamma < G$ be a Zariski dense subgroup. The following are equivalent:
\begin{enumerate}[label=(\roman*)]
\item\label{1} $\Gamma$ has algebraic traces (respectively, integral traces).
\item\label{2} For every faithful representation $\rho: G \to \GL(V)$ of $G$ such that for every $\gamma \in \Gamma$, $\rho(\gamma)$ has algebraic  traces (respectively, integral traces).
\item\label{3} For every faithful representation $\rho: G \to \GL(V)$ of $G$ and every $\gamma \in \Gamma$, the eigenvalues of $\rho(\gamma)$ are algebraic numbers  (respectively, algebraic integers).
\item \label{4} Every faithful representation $\rho: G \to \GL_n(\CC)$ can be conjugated by an element of $\GL_n(\CC)$ so that $\rho(\Gamma) < \GL_n(k)$ (respectively, $\GL_n(O_k)$) for some number field $k$. 
\end{enumerate}
\end{theorem}

\begin{proof}

We prove only the facts for algebraic traces; the facts about integrality follows along the same lines with minor modifications. 

The equivalence between items \ref{2} and \ref{3} follows for a given representation $\rho: \Gamma \to \GL(V)$, because for every $\gamma \in \Gamma$, the coefficients of the characteristic polynomial of $\rho(\gamma)$ can be expressed via Newton's identities (for symmetric polynomals) as a polynomial with rational coefficients on the traces of powers ${\rm tr}\, \rho(\gamma^k)$, and so algebraic traces imply algebraic eigenvalues.

Item \ref{1} is equivalent to \ref{3}, because of the following fact from representation theory: if $\rho_0: G \to \GL(V)$ is an irreducible faithful representation, then any irreducible representation of $G$ appears as a direct summand of the representations obtained by (iterated) tensor products of $\rho_0$ and its dual; see \cite[Ch. I, Prop 3.1]{deligne-hodge-cycles}. Using this fact, and the fact that the eigenvalues of tensor product of representations can be expressed polynomially in eigenvalues of the representations involved, we have that if the elements in the image of $\Gamma$ under the adjoint representation (or any faithful representation) have algebraic eigenvalues, then the same must be true for any faithful irreducible representation.

Item \ref{4} follows because if $\Ad(\Gamma)$ is Zariski dense in $\Ad(G)$, one can construct a faithful representation $\rho: G \to \GL(V)$ where the entries of $\rho(\Gamma)$ (for a well chosen basis of $V$) lie in $k_{\Gamma}$ (This fact appears in the work of Vinberg, see \cite{Vinberg1974Field}; for a proof see Mostow's \cite[section 2.5]{mostow1980remarkable}) and this prove item \ref{4} for $\rho$. By taking tensor product of $\rho$ and its dual and arguing as before, we can suppose that the same is true for any irreducible representation.
\end{proof}








\end{subsection}

\end{section}

\section{Strong ergodicity and boundedness}\label{sec: mainlemma}

The following result will be important in our proofs of the main statements in the introduction. The ideas in its proof come from the works of Venkataramana \cite{Venkataramana1} and Chatterji--Venkataramana \cite{chatterji2009discrete}.

\begin{theorem}\label{thm:bounded}
    Let $G$ be a semi-simple real Lie group with finite center, $\Gamma$ be a discrete subgroup of $G$, and let $\Sigma < \Gamma$ be a subgroup. 
    Let $k$ be a non-Archimedean local field of characteristic 0 and let
    \[
     \rho: \Gamma \to {\rm GL}_n(k)
    \]
    be a homomorphism such that the Zariski closure $H\subset {\rm GL}_n(k)$  of the image of $\rho$ is connected and semi-simple.
    Suppose also that $\Sigma$ acts strongly ergodically on the Furstenberg boundary $G/B$.
    Then if $\rho(\Sigma)$ is pre-compact, then $\rho(\Gamma)$ is pre-compact.
\end{theorem}

\begin{proof}
    Since the Zariski closure $H$ of $\rho(\Gamma)$ in $\mathrm{GL}_n(k)$ is connected and semi-simple, Weyl’s theorem on complete reducibility implies that the finite-dimensional $H$-module $k^n$ decomposes as a direct sum of irreducible $H$-modules:
    \[
        k^n = V_1 \oplus \cdots \oplus V_p.
    \]
    Since $\Gamma < H$ is Zariski dense, this decomposition is also a decomposition of $\Gamma$-modules. Therefore, to prove that $\rho(\Gamma)$ is bounded, it suffices to treat the case where $k^n$ is irreducible under $\Gamma$ (and hence under $H$).

    Furthermore, the connectedness of $H$ implies that any finite-index subgroup $\Gamma_0 \le \Gamma$ is Zariski dense in $H$. It follows that $k^n$ remains irreducible under $\Gamma_0$, so the $\Gamma$-module $k^n$ is strongly irreducible, i.e., no finite-index subgroup preserves a proper nonzero subspace.

    Applying Lemma \ref{map} to $\Gamma < G$, we have a $\Gamma$-equivariant measurable map: $$\phi: G/B \to \Prob(\mathbb{P}^{n-1}(k)).$$ We show that $\phi$ is constant.
    
    By Zimmer \cite[Thm. 3.2.6]{Zimmer}, it follows that ${\rm GL}_n(k)$ acts smoothly on $\Prob(\mathbb{P}^{n-1}(k))$, i.e., every ${\rm GL}_n(k)$-orbit in $\Prob(\mathbb{P}^{n-1}(k))$ is locally closed. Thus, the orbit space
    \[
     \Prob(\mathbb{P}^{n-1}(k))/ {\rm GL}_n(k)
    \]
    is countably seperated.
    Let $$\bar\phi : G/B \to \Prob(\mathbb{P}^{n-1}(k))/ {\rm GL}_n(k)$$ be the composition of $\phi$ with the quotient map  $\Prob(\mathbb{P}^{n-1}(k)) \to \Prob(\mathbb{P}^{n-1}(k))/ {\rm GL}_n(k).$
    Since the codomain of $\bar\phi$ is countably separated and $\Gamma\curvearrowright  G/B$ is ergodic, $\bar\phi$ is essentially constant, i.e., on a co-null set in $G/B$, $\phi$ takes values in a single ${\rm GL}_n(k)$-orbit in $\Prob(\mathbb{P}^{n-1}(k))$. So, we can identify $\phi$ with a $\Gamma$-equivariant measurable map $\phi: G/B \to {\rm GL}_n(k)/L$, where $L< {\rm GL}_n(k)$ is a closed subgroup.

    Since ${\rm GL}_n(k)$ is a locally compact totally disconnected group, there exists a sequence of compact open subgroups $C_n < {\rm GL}_n(k)$ such that $C_n\to \{1\}$ in the topology of Hausdorff convergence. For $n\in\N$, let
    \[
     \Sigma_n \coloneqq \Sigma \cap C_n.
    \]
    Since the image of $\rho(\Sigma)$ is bounded, $\Sigma_n< \Sigma$ is of finite index and thus acts ergodically on $G/B$ by our hypothesis. Since $C_n$ is compact, it acts smoothly on ${\rm GL}_n(k)/L$ and thus by a similar argument as above, the image of $\phi$ lies in a single $C_n$-orbit. Since $C_n\to \{1\}$, it follows that the image of $\phi$ is constant.

    So, $\Gamma$ preserves a probability measure $\mu$ on $\mathbb{P}^{n-1}(k)$.

    Suppose now, to the contrary, that $\rho(\Gamma)$ is unbounded in $H< {\rm GL}_n(k)$. Let
    \[
     [\,\cdot\,]: {\rm GL}_n(k) \to {\rm PGL}_n(k)
    \]
    be the canonical map. Since $H$ is semi-simple, it follows that $[\rho(\Gamma)]$ is unbounded.
    Since $\mu$ is invariant, by Furstenberg's lemma (see \cite[Lem. 3.2.1]{Zimmer}), there exist subspaces
    $V,W\subset k^n$ of smallest possible dimensions such that $\dim V,\dim W <n$ and $\mu$ is supported on $\mathbb{P}(V)\cup \mathbb{P}(W)$. Thus, $\Gamma$ preserves $V\cup W$, which contradicts strong irreducibility of $k^n$.
\end{proof}

We also present an ``Archimedean version'' of the preceeding theorem, although it will not be used in this paper. According to Bader--Furman \cite{BF14}, an action of a countable group on a Lebesgue space $(X,\mu)$ is called {\em isometrically ergodic} if for any action of $\Gamma$ on a separable metric space $M$ by isometries, any $\Gamma$-equivariant map $f : X \to M$ is essentially constant.

\begin{theorem}
    Let $G$, $B$, $\Gamma$, and $\Sigma$ be as in \Cref{thm:bounded}. Let $\rho : \Gamma \to G'$ be a homomorphism into a real algebraic group $G'$, whose image has connected and semisimple Zariski closure. Suppose that $\Sigma$ acts isometrically ergodically on $G/B$. If $\rho(\Sigma)$ is pre-compact, then $\rho(\Gamma)$ is pre-compact.
\end{theorem}

\begin{proof}
    Without loss of generality, we may assume that the Zariski closure of the image of $\rho$ is $G'$. Suppose that $G'$ is noncompact, and let $P'$ be a minimal parabolic subgroup of $G'$. By Benoist--Labourie \cite[Appen.~A]{Benoist-Labourie}, there exists $\gamma \in \Gamma$ such that $\rho(\gamma)$ is regular. Let $\phi : G/B \to \operatorname{Prob}(G'/P')$ be a $\Gamma$-equivariant map as in \Cref{map}. Since $\rho(\Sigma)$ is pre-compact, $G'/P'$ admits a $\rho(\Sigma)$-invariant Riemannian metric. The induced Wasserstein metric on $\operatorname{Prob}(G'/P')$ is invariant and induces the weak topology. Thus isometric ergodicity implies that $\phi$ is essentially constant, with value a $\rho(\Gamma)$-invariant probability measure $\mu$. Since $\rho(\gamma)$ is regular, its recurrent points on $G'/P'$ form a finite set of fixed points. By Poincar\'e recurrence, $\mu$ is supported on this set. Hence $\rho(\Gamma)$ preserves the finite set $\operatorname{supp}(\mu)$. By Zariski density and connectedness, $G'$ fixes each point of this set, contradicting $P' \neq G'$. Thus $G'$ is compact.
\end{proof}

\section{Proof of Theorem \ref{main}}

We will restrict to  the case where $G$ is connected and has trivial center. In this case \cite[Prop. 3.1.6]{Zimmer}, there exists a connected semi-simple algebraic group $\mathbf{G}$ defined over $\Q$ such that $G \cong \mathbf{G}(\R)^0$ (see \cite[Prop. 3.1.6]{Zimmer}).

We will split the proof of Theorem \ref{main} into three Theorems. We will first consider the fact about algebraic and integral traces.

\subsection{Restriction on traces}

\begin{theorem}\label{thm:trace}
Let $\Gamma < G = \bG(\R)^0$ be a finitely generated Zariski dense discrete subgroup of $G$ such that there exists a subgroup $\Sigma< \Gamma$ such that $$\Sigma < \bG(\bar\Q)$$ and that $\Sigma$ acts strongly ergodically on the Furstenberg boundary $G/B$.
Then, $\Gamma$ has algebraic traces.
\end{theorem}

\proof


Since $\Gamma$ is finitely generated, there exists a finitely generated 
$\mathbb Z$–subalgebra $R\subset \mathbb R$ such that 
$\Gamma < \mathbf G(R)$. Let $F=\mathrm{Frac}(R)$ and define 
\[
k := F \cap \bar{\mathbb Q}.
\] 
By 
\cite[Lem.~9.26.11]{stacks-project}, $k$ is a number field.  
 The intersection $A := R \cap \bar{\mathbb Q}$ is a finitely 
generated $\mathbb Z$-algebra, hence there exists $d\in \Z$ such that 
$
A \subset   O_k\left[\frac{1}{d}\right],
$
where $ O_k$ is the ring of integers of $k$.  
Therefore
\begin{equation}\label{eqn:sigma}
    \Sigma < \bG\!\left( O_k\!\left[\tfrac{1}{d}\right] \right).
\end{equation}

Suppose, to the contrary, that there exists $\gamma_0\in \G$ such that $\operatorname{tr}(\Ad(\gamma_0))$ is not algebraic.
Let $\operatorname{Hom}(\Gamma,G)$ be the variety of homomorphisms from $\Gamma$ into $G$.
We consider its subvariety
\begin{equation}\label{eqn:V}
    V := \{ \varphi: \Gamma \to G \mid \, \forall\gamma\in\Sigma, \, \varphi(\gamma) = \gamma \}.
\end{equation}

To see that $V$ is defined over $k$,
fix a finite set of generators $s_1,\dots,s_n$ of $\Gamma$ and a presentation $\Gamma = \<s_1,\dots,s_n \;\vert\; R_i,\, i\in I\>$.
Then, we may realize
\[
V\subset G^{n}
\]
as the zero set of the polynomials in variables $(s_1,\dots,s_n)$ given by:
\begin{enumerate}[label=(\roman*)]
    \item The relations $R_i,\, i\in I$ of $\Gamma$, and 
    \item A countable set of equations determined for each $\gamma \in \Sigma$ (written as a product of the $s_i$'s).
\end{enumerate}
The polynomials in the first category have $\mathbb{Z}$-coefficients whereas the second ones have coefficients in $k$.

We equip $V$ with the euclidean topology.

\begin{lemma}\label{lem1}
    $V$ has positive dimension, $V(\R\cap \bar\Q)$ is dense in $V$, and $$\varphi_0\coloneqq  \operatorname{id}_\Gamma \in V$$ is not an isolated point.
\end{lemma}

\begin{proof}
	Since $V$ is defined over  a number field $k\subset \R$, $V(\R\cap \bar\Q)$ is dense in $V$. 
	Moreover, since $\tr (\Ad(\gamma_0))\not\in \bar\Q$, we have  $\varphi_0\not\in V(\R\cap \bar\Q)$.
    The claim follows from this.
\end{proof}

Now, consider the trace function
\[
 {\rm tr}_{\g_0} : V \longrightarrow \R.
\]
given by ${\rm tr}_{\g_0}(\varphi) \coloneqq \tr(\Ad(\varphi(\gamma_0)))$. By \Cref{lem1}, there exists a sequence $\varphi_n'\in V$ converging to $\varphi_0$ such that $$\varphi_n'(\G) \subset \mathbf{G}({\R\cap \bar\Q}), \quad \text{for all } n\in\N.$$
Since ${\rm tr}_{\g_0}(\varphi_n')$ is algebraic for all $n$ but
${\rm tr}_{\g_0}(\varphi_0)$ is not, and
\[
 \lim_{n\to\infty} \tr_{\gamma_0} (\varphi_n') = {\rm tr}_{\g_0}(\varphi_0),
\]
it follows that the function ${\rm tr}_{\g_0} : V \to\R$ is non-constant near $\varphi_0\in V$.

Consequently, for all large enough prime numbers $p \in\Z$, there exists $\varphi_p \in V$ such that $\varphi_p\to\varphi_0$ as $p\to\infty$ and
\begin{equation}\label{eqn:trace_gamma}
  {\rm tr}_{\g_0}(\varphi_p) = \frac{m}{p},\quad
 \text{for some } m\in \Z\smallsetminus p\Z.
\end{equation}

Similar to \Cref{lem1}, we also have:

\begin{lemma}\label{lem2}
 $\{ \varphi \in V(\R\cap \bar\Q) :\ {\rm tr}_{\g_0}(\varphi) = \frac{m}{p}\}$ is dense in $V \cap {\rm tr}_{\g_0}^{-1} (\frac{m}{p})$.
\end{lemma}

Fix a prime $p\in\Z$ such that $p\nmid d$, where $d$ is as in \eqref{eqn:sigma}.
We can (and will) further assume that \eqref{eqn:trace_gamma} is satisfied by our choice of $p$ and that the homomorphism
\[
 \varphi' \coloneqq \varphi_p
\]
has Zariski dense image.


By \Cref{lem2}, there exists a finite extension $l \subset \C$ of $k$ such that
\[
 \varphi'(\Gamma) < \mathbf{G}(l).
\]
Extending $l$ if necessary, we will further assume that $l$ contains all the eigenvalues
$\lambda_i$ of $\varphi'(\gamma_0)$.

Let $\nu$ denote the $p$-adic valuation on $\Q$, which we extend to a valuation on $l$. 
By the ultrametric inequality we have that
\[
 \min_i\{ \nu (\lambda_i) \} \le  \nu \big(\sum_i \lambda_i\big) = \nu(\tr_{\gamma_0}(\varphi')) <0.
\]
So, we have that
\[
 \max_i | \lambda_i |_\nu >1,
\]
where $| \cdot |_{\nu} \coloneqq e^{-\nu(\cdot)}$ is the norm on $l$.
Let $l_{\nu}$ (a finite extension of $\Q_p$) denotes the completion of $l$ with respect to this norm.

Consider the group $\mathbf{G}(l_{\nu})$, which is a connected and semi-simple $\mathbb{Q}$-group (since $\mathbf{G}$ is assumed to be so).

\begin{lemma}\label{lem:unbounded}
    The homomorphism 
    \begin{equation}\label{eqn:hom}
    \varphi': \Gamma \longrightarrow  \mathbf{G}(l_{\nu})
    \end{equation}
    is unbounded.
\end{lemma}

\begin{proof}
    Since the maximum $|\cdot|_{\nu}$-norm of the eigenvalues of $\gamma_0\in\Gamma$ is strictly greater than $1$, it follows that the subgroup $\<\gamma_0\> < \Gamma$ is unbounded in $\mathbf{G}(l_{\nu})$.
\end{proof}

It follows that $\mathbf{G}(l_{\nu})$ is noncompact.
Consequently, $\mathbf{G}(l_{\nu})$ has a nonzero $l$-rank, i.e., it has a nontrivial $l$-split torus. 

\begin{lemma}\label{lem:bounded}
    $\varphi'(\Sigma) < \mathbf{G}(l_{\nu})$ is bounded.
\end{lemma}

\begin{proof}
This follows from our choice of $p$ (that it does not divide $d$) and the assumption \eqref{eqn:sigma} that $\Sigma< \bG(O_k\left[\frac{1}{d}\right])$.
\end{proof}

\begin{lemma}
    $\varphi'(\Gamma)$ is Zariski dense in $\mathbf{G}(l_{\nu})$.
\end{lemma}

\begin{proof}
    By hypothesis, $\Gamma$ is Zariski dense in $G$. So, we can choose $\varphi'(\Gamma)$ also to be Zariski dense in $G$.
    Since $\mathbf{G}(l_{\nu})$ is a $k$-group,
    a theorem due to Rosenlicht and Chevalley (see Zimmer \cite[Thm. 3.1.9]{Zimmer}) asserts that $\mathbf{G}(k)$ is Zariski dense in $\mathbf{G}(l_{\nu})$.  Hence the conclusion follows.
\end{proof}

As the action of $\Sigma$ on $G/B$ is strongly ergodic while the Zariski closure $\mathbf{G}(l_{\nu})$ of $\varphi'(\Gamma)$ in $\operatorname{GL}_n(l_\nu)$ is connected and semi-simple, we can apply \Cref{thm:bounded} to conclude that $\varphi(\Gamma)$ must be bounded in $\mathbf{G}(l_{\nu})$, which contradicts \Cref{lem:unbounded}.
The proof of \Cref{thm:trace} is now complete.\qed

\begin{theorem}\label{thm:inttrace}
Under the same hypothesis as in \Cref{thm:trace}, suppose further $\Sigma< \bG(O_{\bar \Q})$, where $O_{\bar \Q}$ is the ring of algebraic integers in $\bar Q$. Then, the traces of elements of $\Gamma$ are algebraic integers.
\end{theorem}

\begin{proof}
    Suppose to the contrary that there exists $\gamma_0 \in \Gamma$ such that $\alpha \coloneqq \tr(\gamma_0)$ is not an algebraic integer. By \Cref{thm:trace}, we know that $\alpha\in\bar \Q$. After a small deformation of $\Gamma$ as in \Cref{lem2}, we have a number field $k\subset \R$ such that $\Gamma< \bG(k)$ and $\alpha\in k$.
    

    Since $\alpha \notin O_k$, we can  choose a discrete valuation $\nu$ such that $\nu(\alpha) < 0$. Let $k_\nu$ denote the completion of $k$ with respect to the norm $|\cdot|_\nu$, and let $\varphi: \Gamma \to G(k_\nu)$ the induced homomorphism. By the same argument used in the proof of \Cref{thm:trace}, $\langle \varphi(\gamma_0) \rangle$ and, hence, $\varphi(\Gamma)$ are unbounded subgroups of $\mathbf{G}(k_\nu)$.

    Since $O_k \subset k_v$ is bounded and $\Sigma < \mathbf{G}(O_k)$, $\Sigma < \mathbf{G}(k_v)$ is bounded. Thus, one finishes the proof in the same way as before.
\end{proof}

\begin{remark}\label{rem:Zdense_trace}
    In the statement of \Cref{thm:trace} (resp. \Cref{thm:inttrace}), if we assume that $\Sigma$ is also Zariski dense in $G$, then we could replace the assumption that $\Sigma < \bG(\bar Q)\cap \Gamma$ by the weaker assumption that $\Sigma$ has algebraic traces (resp. integral traces). In that situation, by \Cref{algebraictraces}, we could find a faithful representation $\rho: G \to \GL_n(\C)$ such that $\rho(\Sigma) < \GL_n(k)$ (resp. $\rho(\Sigma) < \GL_n(O_k)$), for a number field $k$. In that situation, $\rho(G)$ will be defined over $k$, and \Cref{thm:trace} applies to get the same conclusion.
\end{remark}

\subsection{Obstruction to Free products}

\begin{theorem}\label{free-product}
    Let $\Sigma < G$ be a finitely generated discrete  subgroup which acts strongly ergodically on $G/B$, and let $\gamma \in G$ non-trivial element. If the group $\Gamma$ generated by $\Sigma$ and $\gamma$ is Zariski dense, then the natural homomorphism
    \begin{equation}\label{def:phi}
        \varphi: \Sigma * \< \gamma \> \longrightarrow G
    \end{equation}
is either not discrete or not faithful.
\end{theorem}

\begin{proof}
We can assume that $\gamma \notin \Sigma $, and arguing by contradiction that $\Sigma * \< \gamma \>$ is faithful and discrete, we can also assume that $\gamma$ has infinite order, and therefore the group $\Gamma$ generated by  $\Sigma$ and $\gamma$ must be isomorphic to $\Sigma * \ZZ$. Furthermore, within $\Sigma * \ZZ$, we may replace $\Sigma$ by $\Sigma' \coloneqq \Sigma * \langle \gamma^2 \rangle$, which satisfies no nontrivial relation with the infinite order element $w = (\gamma s\gamma^{-1}) \gamma^2 (\gamma s\gamma^{-1})^{-1} = \gamma s\gamma^2s^{-1}\gamma^{-1}$, where $\gamma$ is the generator of $\mathbb Z$ and $s$ is a nontrivial element of $\Sigma$. The resulting subgroup $\Sigma' *\langle w\rangle$ of $\Sigma * \ZZ$ has the property that $\Sigma'$ acts strongly ergodically on $G/B$, since it contains $\Sigma$, and $\Sigma'$ is not virtually solvable.
So, we can (and will) further assume that $\Sigma$ is not-virtually solvable.

Let $\mathbf{H}$ be the Zariski closure of $\Sigma$ in $\bG$. Since $\Sigma$ is not virtually solvable, $\mathbf{H}$ must have a Levi factor $\mathbf{L}$ which is simple and noncompact, and the natural homomorphism $\rho: \Sigma \to \mathbf{L}$ has Zariski dense image. 

We can assume that $\rho(\Sigma) < \mathbf{L}(k)$ for a number field $k$ (otherwise, we can look at representation variety of such representations, and find another representation $\rho'$ so that $\rho'(\Sigma) < \mathbf{L}(k)$). From the finitely generated assumption, we can assume that there exists a prime non-Archimedean valuation $\nu$ so that $\mathbf{L}(k_\nu)$ is not compact, and the induced homomorphism $\rho'': \Sigma \to \mathbf{L}(k_\nu)$ has compact image. We can take a non-compact element $\gamma''$ of $\mathbf{L}(k_\nu)$ and extends $\rho''$ to $\Sigma * \ZZ$ by sending the generator of $\ZZ$ to $\gamma''$, and this is clearly contradicts \Cref{thm:bounded}.
\end{proof}

The condition that $\Sigma < G$ is finitely generated is perhaps redundant in the above result. For instance, a minor modification of the argument shows the following:

\begin{theorem}\label{free-product-infgen}
    Let $\Sigma < G$ be a discrete  subgroup which acts strongly ergodically on $G/B$ and let $\gamma \in G$ be an infinite order element. Suppose that $\Sigma \cong \mathbf{F}_\infty$, the free group on (countably) infinitely many generators. Then the natural homomorphism 
    $
        \varphi: \Sigma * \< \gamma \> \longrightarrow G
    $
is either not discrete or not faithful.
\end{theorem}
\begin{proof}[Sketch]
    Realize $\Sigma$ as a subgroup of $\SL_2(\Z)$, which is a lattice in $\mathbf{L}(\R)$, where $\mathbf L = \SL_2$. With this, repeat the proof of \Cref{free-product}.
\end{proof}

In \cite{dey-hurtado-remarks}, it is shown that in any non-compact connected real semi-simple Lie group $G$ with finite center, there exist infinitely many regular loxodromic elements $g_n \in G$, $n\in\N$, freely generating a discrete subgroup $\Gamma$ of $G$ such that their attractive fixed points in $G/B$ form a dense set. Thus $\Gamma$ acts minimally on $G/P$. Now, consider the subgroup $\Gamma_0 \coloneqq \< g_2,g_3,\dots\>$ of $\Gamma$, which, for a similar reason, acts minimally on $G/B$. Moreover, $\Gamma$ is naturally isomorphic to $\<g_1\>*\Gamma_0$. Thus, by \Cref{free-product-infgen}, some finite index subgroup of $\Gamma_0$ cannot act ergodically on $G/B$. So, we have the following:

\begin{corollary}
    Every non-compact connected real semi-simple Lie group $G$ with finite center
contains an infinitely generated discrete subgroup that acts minimally but not ergodically on $G/B$.
\end{corollary}

\section{Irreducible discrete subgroups of $\PSL_2(\mathbb{R}) \times \PSL_2(\R)$}
\label{sec5}

Recall that an {\em irreducible subgroup} in a semi-simple Lie group $G$ is a subgroup that projects densely in each of the simple factors of $G$. For the rest of the section, we let
\[
G \coloneqq \PSL_2(\mathbb{R}) \times \PSL_2(\R),
\quad
H \coloneqq \Delta(\PSL_2(\R)),
\]
and let $\pi_i : G \to \PSL_2(\R)$ denotes the $i$-th coordinate projection for $i = 1$ or $2$.

\subsection{Discrete subgroups intersecting the diagonal in a lattice}

Let $\Gamma$ be a Zariski dense discrete subgroup of $G$, which intersects the diagonal $H = \Delta(\PSL_2(\R))$ in $\Delta(\Lambda)$, where $\Lambda$ is a lattice in $\PSL_2(\R)$. To simplify notation, we will identify $\Lambda$ with $\Delta(\Lambda)$.
Moreover, for a subgroup $J$ of $\PSL_2(\R)$, let $k_J$ denote the field generated by the matrix coefficients of $\operatorname{Ad} J$ with respect to the standard rational basis of $\mathfrak{sl}_2$.
The following statement gives a bound on the degree of $k_{\pi_i(\Gamma)}$ over $k_\Lambda$, for $i=1,2$.

\begin{proposition}\label{prop:degreebound}
 One of the following two alternatives hold:
\begin{enumerate}
 \item $k_{\pi_1(\Gamma)} = k_\Lambda = k_{\pi_2(\Gamma)}$.
 \item $k_{\pi_1(\Gamma)} = k_{\pi_2(\Gamma)}$, and it is a quadratic extension of $k_\Lambda$, and its nontrivial automorphism $\sigma \in \operatorname{Gal}(k/k_\Lambda)$ satsfies $\sigma (\operatorname{Ad} (\pi_1(\gamma))) = \operatorname{Ad} (\pi_2(\gamma))$ for all $\gamma\in\Gamma$.
\end{enumerate}
\end{proposition}

The following lemma is the main input in the proof, whose main idea comes from Bader--Furman \cite[Sec.~3]{BF14}:

\begin{lemma}\label{lem:proj}
Let $\rho : \Gamma \to \PSL_2(\C)$ be a homomorphism such that for all $\lambda \in \Lambda$, $\rho(\lambda) = \lambda$. Then $\rho = \pi_i\vert _\Gamma$ for some $i\in \{1,2\}$.
\end{lemma}

\begin{proof}
 Recall that $\PSL_2(\mathbb R)$ acts on $M = \mathbb S^2$ by conformal automorphisms.
 By \Cref{map}, there exists a measurable $\rho$-equivariant map
 \[
  \phi : G/B = \mathbb S^1 \times \mathbb S^1 \to \operatorname{Prob}(M).
 \]
 We show that there exists some constant $c\in [0,1]$ such that
 \begin{equation}\label{eqn:lem:proj}
     \phi (x,y) = c\, \delta_x + (1-c) \delta_y 
     \qquad \text{for a.e. } (x,y).
 \end{equation}

Suppose that  $\phi(x,y) \left( M^2 \smallsetminus \{x,y\} \right) >0$ for a positive measure set of points $(x,y)$. This set is $\Lambda$-invariant, so either null or conull by ergodicity. Let $\phi_1(x,y)$ be the normalized restriction of $\phi(x,y)$ on $M^2 \smallsetminus \{x,y\}$. So,
\[
 \psi (x,y) = \delta_x \otimes \delta_y \otimes \phi_1(x,y),
 \qquad \text{a.e. } (x,y)
\]
is a measurable $\Lambda$-invariant function from $\mathbb S^1 \times \mathbb S^1 \to \operatorname{Prob}(M^{(3)})$, where $M^{(3)}$ is the space of distinct triples in $M$. However, since the action of $\Lambda$ on $\mathbb S^1 \times \mathbb S^1$ ergodic but $\Lambda$ acts properly discontinuously on $M^{(3)}$, no such function $\psi$ can exist, which gives a contradiction.

Therefore, we must have that for a.e. $(x,y)$
\[
 \phi(x,y) = c(x,y) \delta_x + (1-c(x,y)) \delta_y
\]
for some function $c$. Since $\phi$ is $\Lambda$-invariant, same is $c$, which must be constant by ergodicity. This proves \eqref{eqn:lem:proj}.

By $\rho$-equivariance, we obtain that for all $(\gamma_1,\gamma_2) \in \Gamma$ and a.e. $(x,y)$,
\[
c\, \delta_{\gamma_1 x} + (1-c) \delta_{\gamma_2 y}
= c\, \delta_{\rho(\gamma) x} + (1-c) \delta_{\rho(\gamma) y},
\]
If $c = 0$ (resp. $1$), then we obtain that $\rho = \pi_2\vert_\Gamma$ (resp. $\rho = \pi_1\vert_\Gamma$). Otherwise, $c\in (0,1)$. In this case, the above equality implies that $\rho(\gamma)(x) \in \{ \gamma_1 x, \gamma_2 y\}$, for a.e. $(x,y)$. If $\rho(\gamma)(x) \ne \gamma_1 x$ a.e., then $\rho(\gamma)(x) = \gamma_2 y$ for a positive measure set of points $(x,y) \in \mathbb S^1\times \mathbb S^1$, which is absurd. Thus, $\rho(\gamma) = \gamma_1$, (and similarly, $\rho(\gamma) = \gamma_2$) for all $\gamma \in\Gamma$.
\end{proof}

Now we complete the proof of \Cref{prop:degreebound}:

\begin{proof}
Let $i = 1$ or $2$.
Every embedding $\iota : k_{\pi_1(\Gamma)} \hookrightarrow \C$ fixing $k_\Lambda$ gives a representation $\rho_\iota : \Gamma \to \PSL_2(\C)$, fixing $\Lambda$ pointwise. By the preceding lemma, $\rho_\iota$ is one of the two projections. Thus $k_{\pi_i(\Gamma)}$ has at most two complex embeddings over $k_\Lambda$. Thus  $[k_{\pi_i(\Gamma)} : k_\Lambda] = 1$ or $2$.

If both degrees are one, the first alternative follows.

Otherwise, suppose $[k_{\pi_1(\Gamma)} : k_\Lambda] = 2$ and let $\sigma$ be its nontrivial automorphism. Since $\rho_\sigma\neq\pi_1|_\Gamma$, by the above lemma we must have $\rho_\sigma\neq\pi_2|_\Gamma$.
Hence $\sigma$ intertwines the adjoint coefficients of the two projections, giving 
the conclusion.
\end{proof}

Since the invariant trace field of an arithmetic lattice in $\PSL_2(\R)$ is totally real, we have the following:

\begin{corollary}
If $\Lambda$ is an arithmetic lattice in $\PSL_2(\mathbb R)$, then the invariant trace field of $\pi_i(\Gamma)$, $i=1,2$, is a totally real number field.
\end{corollary}

Now we finish the proof of \Cref{cor:1}:

\begin{proof}
By \Cref{prop:degreebound}, the common field $k=k_{\pi_1(\Gamma)}=k_{\pi_2(\Gamma)}$ is either $\mathbb Q$ or a real quadratic field. Both projections $\pi_1$ and $\pi_2$ take values in $\operatorname{PGL}_2(k)$.
Since $\Gamma$ is finitely generated, after passing to a finite-index subgroup containing $\Lambda$, we can assume that both projections take values in $\operatorname{PSL}_2(k)$.

By \Cref{thm:inttrace} and \Cref{algebraictraces}, 
every matrix $A\in\operatorname{SL}_2(k)$ representing an element of either projection has trace in $O_k$. Indeed, if its eigenvalues are $\lambda,\lambda^{-1}$, then $\lambda^2,\lambda^{-2}$ are eigenvalues of $A$ for the adjoint action, so $\lambda^2+\lambda^{-2} = (\lambda +\lambda^{-1})^2 - 2 \in O_k$, showing that $\operatorname{tr} A = \lambda +\lambda^{-1}$ must be in $O_k$. Since $\pi_i(\Gamma)$ contains $\PSL_2(\Z)$, we can write each entry of $A$ as sums of traces of elements, so each entry of $A$ must lie in $O_k$.

Therefore, it follows that
\[
 \Gamma \subset \PSL_2(O_k) \times \PSL_2(O_k).
\]
If $k=\mathbb Q$, then $\Gamma \subset \PSL_2(\Z) \times \PSL_2(\Z)$.
Else, $k$ is a real quadratic extension of $\mathbb Q$.
In that case, there exists a finite index subgroup $\Lambda \subset \PSL_2(\Z)$ such that
\[
 \Lambda \subset \pi_1 (\Gamma) \subset \PSL_2(O_k).
\]
Since $\Gamma$ is Zariski dense in $G$, $[\pi_1 (\Gamma) : \PSL_2(\Z)] = \infty$.

By a result of Venkataramana \cite[Prop.\ 2.1]{Venkataramana1}, if $\Lambda = \PSL_2(\Z)$, $\pi_1 (\Gamma)$ is a finite index subgroup of $\PSL_2(O_k)$. 
Since we need a version of this result for finite index subgroups $\Lambda$ of $\PSL_2(\Z)$, we give a proof using a recent result of Benoist--Miquel \cite{Benoist-Miquel}:

\begin{lemma}
    For any $v \in \mathbb P^1(k) \smallsetminus \mathbb P^1(\mathbb Q)$, there exists a hyperbolic element $\lambda \in \Lambda$ such that $\lambda$ fixes $v$.
\end{lemma}
\begin{proof}
    In the affine coordinate, write $v = [x:1]$.
    Then $x$ satisfies a quadratic equation, $a x^2 + b x +c = 0$, for some integers $a,b,c$ with discriminant $D = b^2 -4ac$ being a nonnegative square-free integer. Let $(m,n) \ne (1,0)$ be integer solutions to the Pell--Fermat equation $m^2 - Dn^2 = 1$. Then,
    \[
     \underbrace{\begin{bmatrix}
        m-bn & -2cn\\
        2an & m+bn
     \end{bmatrix}}_{\eqqcolon h}
     \begin{bmatrix}
         x \\ 1
     \end{bmatrix}
     = (2anx +bn +m)
     \begin{bmatrix}
         x \\ 1
     \end{bmatrix}.
    \]
    It can be checked that $\det h = 1$, so $h \in \PSL_2(\Z)$, which fixes $v$. Also, $|\tr (h)| = 2|m| > 2$, so $h$ is hyperbolic. Since $\Lambda$ is a finite-index subgroup of $\PSL_2(\Z)$, for a suitable power $N\in\mathbb N$, $h^N \in\Lambda$.
\end{proof}

Let $\gamma \in \pi_1 (\Gamma)$ be any element outside $\PSL_2(\Q)$. Then there exists $w\in \mathbb P^1(\mathbb Q)$ such that $v\coloneqq \gamma w\in \mathbb P^1(k) \smallsetminus \mathbb P^1(\mathbb Q)$. Thus if $u_0 \in \Lambda$ is an unipotent matrix fixing $w$, then $u\coloneqq \gamma u_0 \gamma^{-1}$ fixes $v$. Let $\lambda \in \Lambda$ be a hyperbolic matrix fixing $v$. Then
\[
 \langle \lambda u \lambda^{-1}, u \rangle
\]
is an indecomposable lattice in the maximal unipotent subgroup of $G$ fixing $v$. Thus by Benoist--Miquel \cite[Thm. 2.16]{Benoist-Miquel}, $\pi_1 (\Gamma)$ is a finite index subgroup of $\PSL_2(O_k)$.

Since the second projection $\pi_2$ is a conjugate of $\pi_1$ by the element $\sigma$ in \Cref{prop:degreebound}, $\Gamma$ is an irreducible arithmetic lattice in $G$.
\end{proof}

\subsection{Irreducibility vs. ergodicity}

The following statement shows that an affirmative answer to \Cref{ques:irr} will have a number of remarkable consequences, including a negative answer to \Cref{ques:free}:

\begin{theorem}\label{irreducibleergodic}
 If every irreducible discrete subgroup of $G$ acts ergodically on the Furstenberg boundary $\S^1\times \S^1$, then the following hold:
    \begin{enumerate}[label=(\roman*)]
        \item No irreducible discrete subgroup $\Gamma$ of $G$ is Gromov hyperbolic.
        \item Let $\Gamma$ be a hyperbolic group, which embeds discretely as a subgroup of $G$. Then $\Gamma$ is virtually either a free group or a surface group.
        \item If $M$ is a finite volume hyperbolic 3-manifold, then $\pi_1(M)$ does not admit any discrete and faithful homomorphism into $G$.
    \end{enumerate}
\end{theorem}

Regarding the final item, note that for cohomological reasons, if $M$ is a closed hyperbolic $n$-manifold where $n \ge 4$, then $\pi_1(M)$ does not embed discretely as a subgroup of $G$.

We will make use of the following elementary observations:

\begin{lemma}\label{lem:elliptic}
    Let $\Gamma$ be a finitely generated dense subgroup of $\PSL_2(\R)$. Then, $\Gamma$ contains an infinite order elliptic element.
\end{lemma}

\begin{proof}
    Since $\Gamma$ is finitely generated, Selberg’s Lemma \cite{selberg-lemma} implies that $\Gamma$ contains a finite-index, torsion-free subgroup $\Gamma'$ (which must be dense in $\PSL_2(\mathbb{R})$ since $\Gamma$ is).

    Recall that elliptic elements $g$ in $\PSL_2(\mathbb{R})$ are characterized by the trace condition $|\operatorname{tr}(g)| < 2$. Since $\Gamma'$ is dense in $\PSL_2(\mathbb{R})$, it follows that $\{\operatorname{tr}(\gamma) :\ \gamma \in \Gamma'\}$ is dense in $\mathbb{R}$. In particular, $\Gamma'$ contains an elliptic element, which must be of infinite order since $\Gamma'$ is torsion-free.
\end{proof}

\begin{proposition}\label{prop: niceelements} If $\Gamma < G$ is an irreducible discrete subgroup, then $\Gamma$  contains two elements $\gamma_1 = (e_1, h_1)$ and $\gamma_2 = (h_2, e_2)$, where $e_i$'s (resp. $h_i$'s) are infinite-order elliptic (resp. hyperbolic) elements in $\PSL_2(\mathbb{R})$
\end{proposition}

\begin{proof}
Using the fact that the projection of $\G$ in the first factor of $G$ is dense, it follows by \Cref{lem:elliptic} that $\G$ contains an element $\g_1 = (e_1,h_1)$, where $e_1\in\PSL_2(\R)$ is an infinite order elliptic element.
Since $\G$ is discrete in $G$, $\<h_1\>$ is an infinite discrete cyclic subgroup of $\PSL_2(\R)$; so, $h_1$ must be a parabolic or hyperbolic element. We show that $h_1$ can be taken to be a hyperbolic element: indeed, suppose that $h_1$ is parabolic. Consider the group $\Delta$ generated by $\g_1$ and a conjugate $\g_1' = (e_1',h_1')$ of $\g_1$ by an element of $\G$ of the form $(g_1,g_2)$, where 
$g_2$ is an infinite order elliptic element. Thus, $g_2 x \ne x$, where $x$ denotes the unique fixed point of $h_1$ in $S^1$. Applying the classical ping-pong argument, it is easily seen that for all large enough $n\in\N$, we have the following:
\begin{enumerate}
    \item $\<h_1^n,h_1'^n\>$ is a discrete subgroup of $\PSL_2(\R)$, freely generated by $h_1^n$ and $h_1'^n$.
    \item The parabolic elements of $\<h_1^n,h_1'^n\>$ can be conjugated into $\<h_1\>$ and $\<h_1'\>$.
\end{enumerate}
Fix such an $n$. Considering the traces, it can be seen that $e_1^n(e_2^n)^{k}$ is elliptic for infinitely many values of $k\in\N$. Let $k$ be such a number so that $e_1^n(e_2^n)^{k}$ is elliptic (which is again of infinite order). Then, for the element $(e_1^n(e_2^n)^{k}, h_1^n(h_2^n)^{k})\in \G$, the first component is elliptic, whereas the second one is hyperbolic (due to conditions (i) and (ii) above).

Similarly, $\G$ contains an element $\g_2 = (h_2,e_2)$, where $h_2$ (resp. $e_2$) are infinite order hyperbolic (resp. elliptic) elements of $\PSL_2(\R)$.    
\end{proof}

We remark that any discrete subgroup $G$ generated by two elements $\gamma_1$ and $\gamma_2$ of the form given by \Cref{prop: niceelements} is necessarily irreducible.

We can now start the proof of \Cref{irreducibleergodic}.

\begin{proof}[Proof of  \Cref{irreducibleergodic} Item (i)]

Suppose, to the contrary, that $\G$ is a discrete, irreducible, hyperbolic subgroup of $G$. Since $\G$ is finitely generated, by applying Selberg's Lemma, we can assume that $\Gamma$ and its projections to each factor of $G$ are torsion-free.

Clearly, no powers of $\g_1$ and $\g_2$ are conjugate in $\G$ (since these are not conjugate in $\PSL_2(\R)$).
Since $\G$ is a torsion-free hyperbolic group, again by a ping-pong argument, it can be seen that for all large enough $n\in\N$, the natural homomorphism $\rho: \<\g_1^n\> *\<\g_2^n\> \to \G$ is injective. Note that the image of $\rho$ is also irreducible in $G$.

Thus, the proof now reduces to case when $\G$ is the free group $F_2$ generated by $\gamma_1$ and $\gamma_2$, as in \Cref{prop: niceelements}.

In this case $\g_1^2$ and $\g_2^2$ also generate an irreducible discrete subgroup $F$ of $G$, which by assumption acts ergodically, which moreover does not contain the element $\g_1\g_2$. Observe that $\g_1^2$ and $\g_2^2$ and $g_1g_2$ generate the free group $F_3$.

It is easy to see that any finite index subgroup $F$ contains some powers of $\g_1$ and $\g_2$ and, hence, and so by our assumption on irreducible groups acting ergodically, we have that $F$ acts strongly ergodically. 
    
Thus, by \Cref{free-product}, the natural homomorphism
    \[
     \rho: F* \<g_1g_2\> \to G
    \]
    cannot be discrete and faithful, which gives us a contradiction.
\end{proof}

Before proving Item (ii), we need the following:

\begin{lemma}\label{prop:HtimesR}
  Let $\Gamma$ be a torsion-free non-elementary hyperbolic group, which embeds discretely as a subgroup of
  $\PSL_2(\R)\times P$, where $P$ is a solvable Lie group. Then the projection of $\Gamma$ to the $\PSL_2(\R)$-factor is discrete and faithful.

  In particular, $\Gamma$ is either a free group or a surface group.
\end{lemma}

\begin{proof}
Let
\[
\rho: \Gamma \rightarrow \PSL_2(\R)\times P
\]
be a discrete and faithful homomorphism.
Let ${\rm pr}_1$ and ${\rm pr}_2$ denote the projections of $\PSL_2(\R)\times P$ onto the first and second factors, respectively. We will identify $\Gamma$ with its image $\rho(\Gamma)$.

Set
\[
N \coloneqq \ker\big({\rm pr}_2|_{\Gamma}\big) \triangleleft \Gamma.
\]
Since $\Gamma$ is discrete in $\PSL_2(\R)\times P$, the subgroup $N < \PSL_2(\R)\times\{e\}$ is
also discrete. In particular, ${\rm pr}_1|_N : N \to \PSL_2(\R)$ is discrete and faithful.

Because $P$ is solvable, the quotient $\Gamma/N \cong {\rm pr}_2(\Gamma)$ is
solvable. Since $\Gamma$ is not solvable, $N$ must be non-trivial, hence infinite, and hence nonelementary.
Thus ${\rm pr}_1(N)$ is a nonelementary discrete subgroup of $\PSL_2(\R)$.

Since ${\rm pr}_1(N)$ is normal in
${\rm pr}_1(\Gamma)$, ${\rm pr}_1(\Gamma)$ must be discrete in $\PSL_2(\R)$; otherwise, the identity component $H$ of its closure would have positive dimension, would centralize ${\rm pr}_1(N)$ (by discreteness), and this would be impossible since ${\rm pr}_1(N)$ is a nonelementary discrete subgroup of $\PSL_2(\R)$.

Now consider
\[
K := \ker\big({\rm pr}_1|_{\Gamma}\big) \triangleleft \Gamma.
\]
Since $K$ injects into $P$, $K$ is solvable.
Since $\Gamma$ is a torsion-free nonelementary hyperbolic group, any nontrivial normal subgroup of $\Gamma$ is nonsolvable. Thus $K$ must be trivial, so ${\rm pr}_1|_{\Gamma}$ is faithful.
\end{proof}

Now we return to the proof of  \Cref{irreducibleergodic}, item (ii). We will essentially show that if $\Gamma$ is a hyperbolic discrete subgroup of $G$, which is virtually neither a free group nor a surface group, then $\Gamma$ must be irreducible.

\begin{proof}[Proof of  \Cref{irreducibleergodic} Item (ii)]
    Suppose, to the contrary, that there exists a hyperbolic group $\Gamma$, which is virtually neither a free group nor a surface group, such that it admits a discrete and faithful homomorphism
    \[
     \rho: \Gamma \to G.
    \]
    Using Selberg's Lemma, we may assume that $\Gamma$ is torsion-free.
    We also will identify $\Gamma$ with its image $\rho(\Gamma)$.\
    
    Note that the projections of $\G$ into the any of the factors of $G$ cannot be bounded, since otherwise, due to discreteness in $G$, $\Gamma$ would act properly discontinuously on $\mathbb{H}^2$ through the projection into the other factor -- this would be impossible since then $\Gamma$ would be a free group or a surface group.
    Moreover, these images cannot also lie in a parabolic subgroup of $\PSL_2(\R)$, as rulled out by \Cref{prop:HtimesR}.
    Thus, we can assume that the projection of $\Gamma$ to each factor is Zariski dense.
    Thus, we have the following three mutually exclusive possibilities:

    \medskip\noindent
    {\bf Case 1.} Suppose that the image of $\Gamma$ in each factor is discrete. In this case,  $\Gamma$ virtually embeds into a direct product of two torsion-free finitely generated discrete subgroups of $\PSL_2(\R)$.
    Since every hyperbolic group is finitely presented \cite{gromov-hyperbolic}, $\Gamma$ is of type $\mathrm{FP}_2$ and thus by \cite[Thm. A]{MR1934013}, $\Gamma$ is virtually a direct product of at most two surface or free groups. Since $\Gamma$ does not contain any $\Z^2$, it follows that $\Gamma$ is either a surface group or a free group, which is a contradiction.

    \medskip\noindent
    {\bf Case 2.} Suppose that $\Gamma$ has a discrete projection in one factor, say the first, and an indiscrete projection in the other (i.e., the second). 
    Since the projection of $\G$ to the second factor is also Zariski dense, this projection is (topologically) dense.

    Let $N$ denote the kernel of the projection of $\Gamma$ into the first factor. Since $\G$ is discrete in $G$, $N$ acts properly discontinuously on the hyperbolic plane $\mathbb{H}^2$ through the projection into the second factor. Since the image of $\Gamma$ is dense in the second factor, we see that the normalizer of $N$, where $N$ is viewed as a discrete subgroup of the second factor through the projection, is dense in $\PSL_2(\R)$. But the normalizer of a discrete subgroup of $\PSL_2(\R)$ must be closed. Hence, $N$ must be a normal subgroup of $\PSL_2(\R)$. So, $N$ is central (hence trivial since $\Gamma$ is torsion-free) and, consequently, $\Gamma$ embeds discretely in $\PSL_2(\mathbb{R})$ through the first projection, which is impossible.

    \medskip\noindent
    {\bf Case 3.}
    Suppose that $\Gamma$ to both factors of $G$ are indiscrete. Since the two projections are Zariski dense, both projections are topologically dense and hence $\Gamma$ is irreducible. Thus, if all irreducible discrete subgroups of $G$ were to act ergodically act on $\S^1\times \S^1$, then by Item (i) of \Cref{irreducibleergodic}, $\Gamma$ cannot be hyperbolic, again a contradiction.
\end{proof}

\begin{proof}[Proof of  \Cref{irreducibleergodic} Item (iii)]
 If $M$ is closed, then one can argue as in (ii) to obtain the result. In general, it is known that $\pi_1(M)$ contains a plethora of quasifuchsian surface groups \cite{kahn-markovic,cooper-futer,kahn-wright}. Hence one can find a subgroup isomorphic to $\pi_1(S) * \mathbb{Z}$ inside $\pi_1(M)$, where $S$ is a closed hyperbolic surface, and the conclusion then follows from (ii).
\end{proof}

\section{Proofs of \Cref{diagonallattice} and \Cref{hyperbolictraces}}

\subsection{Proof of \Cref{diagonallattice}}

If $\Sigma$ is a lattice in a semi-simple Lie group with finite center, the Howe-Moore ergodicity implies that $\Sigma$ acts ergodically on $G/B \times G/B$ as  this space is isomorphic to $G/A$, where $A$ is the Cartan subgroup. Therefore, $\Delta(\Sigma)$ acts ergodically on $G/B\times G/B = (G\times G)/(B \times B)$. Therefore Theorem \ref{main} implies \Cref{diagonallattice}.

\subsection{Proof of \Cref{hyperbolictraces}}
If $\Gamma$ is a lattice in $G = {\rm SO}(n,1) = \text{Isom}(\mathbb{H}^n)$, then $\Gamma$ is Zariski dense on $G$. In fact, since the subgroup $\Gamma_0$ acts ergodically on the boundary $\mathbb{S}^{n-1}$ of  $G$, it has a full limit set in $\mathbb{S}^n$, and hence $\Gamma_0$ is also Zariski dense in $G$.
In this situation, if $\Gamma_0$ has integral (resp. algebraic) traces, then \Cref{thm:inttrace} (resp. \Cref{thm:trace}), together with \Cref{rem:Zdense_trace}, shows that $\Gamma$ has integral (resp. algebraic) traces.

\appendix 

\section{A geometric proof of \Cref{diagonallattice} in the cocompact case}\label{geomproof}

In this section, we provide a geometric proof of part of \Cref{diagonallattice} in the case where the lattice is cocompact. This proof is more geometric, and does not use boundary maps.

Let $G$ denote a connected real semi-simple Lie group with finite center and without compact factors. We let $K$ be a maximal compact subgroup, let $A$ be a Cartan subgroup and $P$ a minimal parabolic subgroup containing $A$, and let $N$ be the unipotent radical of $P$ so $P = AN$. We have the Iwasawa decomposition $G = KP$. Let $\Lambda < G$ be a cocompact lattice and let $\Delta: G \to G \times G$ be the diagonal embedding. 

\begin{theorem}
Let $\Gamma$ be a Zariski dense discrete subgroup of $G \times G$ that contains $\Delta(\Lambda)$, then $\Gamma$ is not isomorphic to a non-trivial free product.
\end{theorem}

\begin{proof}
The proof is by contradiction. We assume that $\Gamma$ is isomorphic to a free product, and by considering an infinite order element not in $\Delta(\Lambda)$, and possibly replacing $\Gamma$ with a subgroup, we can assume that  $\Gamma = \Delta(\Lambda) * \mathbb{Z}$.

We will use the following elementary proposition which follows from the Zariski density of $\Gamma$. 

\begin{proposition}\label{choice}
There exists $(\gamma_1,\gamma_2) \in \Gamma$ such that:

\begin{enumerate}
\item The group generated by $\Delta(\Lambda)$ and $(\gamma_1,\gamma_2)$ is naturally isomorphic to $\Delta(\Lambda) * \Z$ 
\item There exists $kP \in G/B$, where $k\in K$, such that $\gamma_1 kP = \gamma_2 kP$ 
\end{enumerate}

\end{proposition}

Choose $(\gamma_1, \gamma_2)$ and $kP$ as in Proposition \ref{choice}. We write $\gamma_i k = k' p_i$ for some $k' \in K$ and $p_1, p_2 \in P$.
Take  $a$ to be an element of $A$ so that $N$ is the stable horospherical subgroup of $P$, that is  $$N = \{ n \in G, \text{ such that } \lim_{t \to \infty} d(a^tna^{-t}, 1) = 0 \}.$$
Because of our assumption that $\Lambda$ is a uniform lattice, there exists a compact set $C \subset G$, such that for every $x \in G$ the intersection $\Gamma \cap Cx \neq \emptyset$. 

Therefore for every $t > 0$, there exists $c_t \in C$ and $\lambda_t \in \Lambda$ such that $c_ta^tk = \lambda_t$, and similarly there exists $c'_t \in C$ and $\lambda'_t$ such that $c'_ta^t.k' = \lambda'_t$.

\begin{proposition}\label{compact}
There exists a compact set $C_1 \subset G \times G$, such that for every $t>0$, the element $(\lambda'_t, \lambda'_t)^{-1} (\gamma_1, \gamma_2) (\lambda_t, \lambda_t) \in C_1$.
\end{proposition}

\begin{proof}

We have the following equation:
\begin{align*}
(\lambda'_t)^{-1} \gamma_i \lambda_t &=  (k_1a^{-t}(c'_t)^{-1})^{-1} \gamma_i (ka^{-t}c_t^{-1}) \\
&=   c'_t a^t (k_1^{-1}\gamma_ik) a^{-t} c_t^{-1}\\
&=   c'_t (a^t p_i a^{-t}) c_t
\end{align*}

Observe that by our choice of $a$, the element $a^t p_i a^{-t}$ lies in a bounded set independent of $G$ and as $c_t, c'_t$ lie in the compact set $C$, we are done.

\end{proof}

The discreteness of $\Gamma$ implies that  the set of elements $$S := \{ (\lambda'_t, \lambda'_t)^{-1} (\gamma_1, \gamma_2) (\lambda_t, \lambda_t) \text{ for } t > 0\}$$ in \ref{compact} is a finite collection of elements in $
\Gamma$, but as $\lambda_t, \lambda'_t$ diverge in $G$ as $t \to \infty$, there exists $s,t >0$, such that $\lambda_s \neq \lambda_t$ and $\lambda'_s \neq \lambda'_t$ such that
$$  (\lambda'_t, \lambda'_t)^{-1} (\gamma_1, \gamma_2) (\lambda_t, \lambda_t) =  (\lambda'_s, \lambda'_s)^{-1} (\gamma_1, \gamma_2) (\lambda_s, \lambda_s), $$
which contradicts that  $\Delta(\Lambda)$ and $(\gamma_1, \gamma_2)$ generate a free product.
\end{proof}

\section{Some connections between irreducible groups and diophantine properties of Lie groups}

The main point of this appendix is to show that proving certain irreducible groups are discrete is at least as hard as constructing new examples of dense subgroups satisfying a non-abelian diophantine property.

Let $G$ be a semi-simple Lie group, and $\Gamma < G$, a finitely generated dense subgroup, let $S$ be a finite generating set of $\Gamma$. Fix $d$ to be a left invariant metric in $G$, and let $B_G(g,r)$ denote the metric ball centered at $g \in G$ of radius $r\geq 0$.

\begin{definition} We say that $\Gamma$ is \emph{strongly diophantine} if there exists $c_1, c_2>0$ such that $B_G(1_G,c_1e^{-c_2n} )\cap S^n = \{1_G\}$. We say that $\Gamma$ is \emph{weakly diophantine} if there exists $c_1, c_2 >0$
$$\limsup_{n \to \infty} \frac{\log |B_G(1_G,c_1e^{-c_2n} )\cap S^n | }{ \log |S^n| }  < 1$$
\end{definition}

Observe that this definition is independent of the generating set $S$. If $S$ consists of elements with algebraic entries (under some faithful representation of $G$), then $\Gamma$ is strongly diophantine.\\ 

If $G$  is compact and simple, Benoist-DeSaxce \cite{benoist2016spectral} (following work of Bourgain and Gamburd for ${\rm SU}(n)$ \cite{bourgain2008spectral, bourgain2012spectral}) have shown that being almost-diophantine (a condition related to being weakly-diophantine) is equivalent to having a spectral gap for the averaging operator $T: L_0^2(G) \to L_0^2(G)$ given by $$T(f)= \frac{1}{|S|} \sum_{s \in S} f\circ s$$

The existence of the spectral gap for this operator (i.e. $\|T\|_{op} <1$) has been conjectured by Sarnak to hold in the case when $S = S^{-1}$ and the generators of $S$ are Haar generic, but the only known examples (as far as we know) are when $S$ consists of elements with algebraic entries (up to conjugation).\\

Consider the following irreducible subgroup:\\

Let $\theta \in [0,2\pi]$ be an irrational multiple of $2\pi$, and $\lambda >0$, define the matrices

\begin{align*}
R_{\theta} := \begin{bmatrix} cos(\theta) & -sin(\theta) \\ sin(\theta) & \cos(\theta)
\end{bmatrix} \ \ \text{and} \ \ 
A_{\lambda} :=  \begin{bmatrix} e^{\lambda} & 0 \\ 0 & e^{-\lambda}
\end{bmatrix}   
\end{align*}

Define the subgroup of $\SL_2(\RR)\times \SL_2(\RR)$ $$\Gamma_{\theta, \lambda} := \langle R_{\theta} \times A_{\lambda}, A_{\lambda} \times R_{\theta} \rangle$$ 

Let $\Delta_{\theta, \lambda}$ be the projection of $\Gamma_{\theta, \lambda}$ in the first factor (or second factor, as they are the same). 

The following holds:

\begin{proposition} For every $\theta \in [0,2\pi]$ irrational multiple of $\pi$, and $\lambda >0$, the group  $\Gamma_{\theta, \lambda}$ is an irreducible subgroup of $\SL_2(\RR) \times \SL_2(\RR)$. Moreover for Lebesgue a.e. $\theta, \lambda$, the group  $\Gamma_{\theta, \lambda}$ is free.
\end{proposition}

\begin{question} Does for Lebesgue a.e. $\theta$, there exists $\lambda_{\theta}$, such that for every $\lambda \in [\lambda_\theta, \infty)$, the group $\Gamma_{\theta, \lambda}$ is discrete?
\end{question}

The reason why one might expect something like this to hold is that for a non-trivial word $w(x,y)$ in the free group $F_2$, it seems plausible that at least one among $w(R_{\theta}, A_{\lambda})$ or  $w( A_{\lambda}, R_{\theta})$ must be far from the identity. Nonetheless, this seems difficult to achieve because of the following:

\begin{proposition} If $\Gamma_{\theta, \lambda}$ is discrete, then the projection $\Delta_{\theta, \lambda}$ into one (both) of the $\SL_2(\R)$ factors is strongly diophantine.
\end{proposition}

\begin{proof}
Suppose that $\Delta_{\theta, \lambda}$ is not strongly diophantine, then if we let $C := 3\log \|A_{\lambda}\|  + 1$, there exists a divergent sequence $\{n_k\}_k>0$, and words $w_k$ of length $n_k$ in $S$ such that 

$$ 0 < d(\text{Id}, w_k(R_{\theta}, A_{\lambda})) < e^{-C n_k} $$

Consider the word $w'_k = [w_k(x,y), x^lw_k(y,x)x^{-l}]$ where $[,]$ denotes the commutator, and $l$ is a fixed integer to be determined later. We show that both projections tend to the identity as $k$ goes to infinity. 

The projection in the first factor is:
$$[w_k(R_{\theta},A_{\lambda}), R_{\theta}^{l}w_k(A_{\lambda},R_{\theta})R_{\theta}^{-l}]$$

Observe that $\|w_k(A_{\lambda},R_{\theta})\| < e^{\frac{C}{3}n_k}$, where $\|.\|$ denotes the operator norm (which is submultiplicative), and so one can check easily that:

$$ d(1, [w_k(R_{\theta},A_{\lambda}), R_{\theta}^{l}w_k(A_{\lambda},R_{\theta})R_{\theta}^{-l}] ) < M e^{-\frac{C}{3}n_k}$$ for some constant $M>0$ (independent of $k$, but possibly depending on $l$).

One shows similarly the second projection $$[w_k(A_{\lambda}, R_{\theta}),A^{l}w_k(R_{\theta},A_{\lambda})A^{-l}]$$ satisfy a similar inequality and so if $\Gamma_{\lambda, \theta}$ were discrete, then for $k$ large enough, we have that both projections are trivial, and so for $k$ large and $l = 1,2,3$, we have $w_k(R_{\theta},A_{\lambda})$ and $R_{\theta}^{l}w_k(A_{\lambda},R_{\theta})R_{\theta}^{-l}$ commute, and also $ w_k(A_{\lambda}, R_{\theta})$ and $ A^{l}w_k(R_{\theta},A_{\lambda})A^{-l}$ commute, by considering the action by these matrices in the hyperbolic plane, one can see that this only can happen if $A_{\lambda}$ and $R_{\theta}$ commute, a contradiction, and therefore $\Gamma_{\lambda, \theta}$ is not discrete.
\end{proof}

\bibliography{bibliography.bib} 

@inproceedings{BF14,
	author = {Bader, Uri and Furman, Alex},
	booktitle = {Proceedings of the {I}nternational {C}ongress of {M}athematicians---{S}eoul 2014. {V}ol. {III}},
	isbn = {978-89-6105-806-3; 978-89-6105-803-2},
	mrclass = {37D25 (22E40)},
	mrnumber = {3729019},
	mrreviewer = {Lifan\ Guan},
	pages = {71--96},
	publisher = {Kyung Moon Sa, Seoul},
	title = {Boundaries, rigidity of representations, and {L}yapunov exponents},
	year = {2014}}

@article{DFH26,
	archiveprefix = {arXiv},
	author = {Subhadip Dey and Miko{\l}aj Fr{\k{a}}czyk and Sebastian Hurtado},
	eprint = {2608.27274},
	journal = {Preprint, arXiv 2608.27274},
	primaryclass = {math.DS},
	title = {On some aspects of discrete groups acting ergodically on the boundary},
	url = {https://arxiv.org/abs/2608.27274},
	year = {2026}}

@article {Benoist-Miquel,
    AUTHOR = {Benoist, Yves and Miquel, S\'ebastien},
     TITLE = {Arithmeticity of discrete subgroups containing horospherical
              lattices},
   JOURNAL = {Duke Math. J.},
  FJOURNAL = {Duke Mathematical Journal},
    VOLUME = {169},
      YEAR = {2020},
    NUMBER = {8},
     PAGES = {1485--1539},
      ISSN = {0012-7094,1547-7398},
   MRCLASS = {22E40 (11F06 20H05)},
  MRNUMBER = {4101737},
MRREVIEWER = {O.\ V.\ Shvartsman},
       DOI = {10.1215/00127094-2019-0082},
       URL = {https://doi.org/10.1215/00127094-2019-0082},
}

@article {Benoist-Labourie,
    AUTHOR = {Benoist, Yves and Labourie, Fran\c{c}ois},
     TITLE = {Sur les diff\'{e}omorphismes d'{A}nosov affines \`a
              feuilletages stable et instable diff\'{e}rentiables},
   JOURNAL = {Invent. Math.},
  FJOURNAL = {Inventiones Mathematicae},
    VOLUME = {111},
      YEAR = {1993},
    NUMBER = {2},
     PAGES = {285--308},
      ISSN = {0020-9910,1432-1297},
   MRCLASS = {58F15},
  MRNUMBER = {1198811},
MRREVIEWER = {M.\ Rees},
       DOI = {10.1007/BF01231289},
       URL = {https://doi.org/10.1007/BF01231289},
}

@article{agol2004tameness,
	author = {Agol, Ian},
	journal = {arXiv preprint math/0405568},
	title = {Tameness of hyperbolic 3-manifolds},
	year = {2004}}

@article{bader-miller-fisher-stover,
	author = {Bader, Uri and Fisher, David and Miller, Nicholas and Stover, Matthew},
	doi = {10.4007/annals.2021.193.3.4},
	fjournal = {Annals of Mathematics. Second Series},
	issn = {0003-486X,1939-8980},
	journal = {Ann. of Math. (2)},
	mrclass = {22E40},
	mrnumber = {4250391},
	mrreviewer = {Thilo\ Kuessner},
	number = {3},
	pages = {837--861},
	title = {Arithmeticity, superrigidity, and totally geodesic submanifolds},
	url = {https://doi.org/10.4007/annals.2021.193.3.4},
	volume = {193},
	year = {2021}}

@article{bass-lubotzky,
	author = {Bass, Hyman and Lubotzky, Alexander},
	doi = {10.2307/121131},
	fjournal = {Annals of Mathematics. Second Series},
	issn = {0003-486X,1939-8980},
	journal = {Ann. of Math. (2)},
	mrclass = {20F65 (22E40)},
	mrnumber = {1779566},
	mrreviewer = {Roger\ C.\ Alperin},
	number = {3},
	pages = {1151--1173},
	title = {Nonarithmetic superrigid groups: counterexamples to {P}latonov's conjecture},
	url = {https://doi.org/10.2307/121131},
	volume = {151},
	year = {2000}}

@article{bourgain2012spectral,
	author = {Bourgain, Jean and Gamburd, Alex},
	journal = {Journal of the European Mathematical Society (EMS Publishing)},
	number = {5},
	title = {A spectral gap theorem in SU.(d).},
	volume = {14},
	year = {2012}}

@article{benoist2016spectral,
	author = {Benoist, Yves and de Saxc{\'e}, Nicolas},
	journal = {Inventiones mathematicae},
	number = {2},
	pages = {337--361},
	publisher = {Springer},
	title = {A spectral gap theorem in simple Lie groups},
	volume = {205},
	year = {2016}}

@article{bourgain2008spectral,
	author = {Bourgain, Jean and Gamburd, Alex},
	journal = {Inventiones mathematicae},
	number = {1},
	pages = {83--121},
	publisher = {Springer},
	title = {On the spectral gap for finitely-generated subgroups of SU (2)},
	volume = {171},
	year = {2008}}

@article{brody2023greenberg,
	author = {Nic Brody and David Fisher and Mahan Mj and Wouter {van Limbeek}},
	doi = {https://doi.org/10.1016/j.exmath.2025.125736},
	issn = {0723-0869},
	journal = {Expositiones Mathematicae},
	pages = {125736},
	title = {Greenberg--Shalom's commensurator hypothesis and applications},
	url = {https://www.sciencedirect.com/science/article/pii/S072308692500091X},
	year = {2025}}

@article{calegari-gabai,
	author = {Calegari, Danny and Gabai, David},
	doi = {10.1090/S0894-0347-05-00513-8},
	fjournal = {Journal of the American Mathematical Society},
	issn = {0894-0347,1088-6834},
	journal = {J. Amer. Math. Soc.},
	mrclass = {57M50 (30F40 57N10)},
	mrnumber = {2188131},
	mrreviewer = {Bruno\ P.\ Zimmermann},
	number = {2},
	pages = {385--446},
	title = {Shrinkwrapping and the taming of hyperbolic 3-manifolds},
	url = {https://doi.org/10.1090/S0894-0347-05-00513-8},
	volume = {19},
	year = {2006}}

@article{cooper-futer,
	author = {Cooper, Daryl and Futer, David},
	doi = {10.2140/gt.2019.23.241},
	fjournal = {Geometry \& Topology},
	issn = {1465-3060,1364-0380},
	journal = {Geom. Topol.},
	mrclass = {20F65 (20H10 30F40 57M50)},
	mrnumber = {3921320},
	mrreviewer = {Yasushi\ Yamashita},
	number = {1},
	pages = {241--298},
	title = {Ubiquitous quasi-{F}uchsian surfaces in cusped hyperbolic 3-manifolds},
	url = {https://doi.org/10.2140/gt.2019.23.241},
	volume = {23},
	year = {2019}}

@book{deligne-hodge-cycles,
	author = {Deligne, Pierre and Milne, James S. and Ogus, Arthur and Shih, Kuang-yen},
	isbn = {3-540-11174-3},
	mrclass = {14Kxx (10D25 12A67 14A20 14F30 14K22)},
	mrnumber = {654325},
	pages = {ii+414},
	publisher = {Springer-Verlag, Berlin-New York},
	series = {Lecture Notes in Mathematics},
	title = {Hodge cycles, motives, and {S}himura varieties},
	volume = {900},
	year = {1982}}

@article{dey-hurtado-remarks,
	author = {Dey, Subhadip and Hurtado, Sebastian},
	doi = {10.1093/imrn/rnaf268},
	fjournal = {International Mathematics Research Notices. IMRN},
	issn = {1073-7928,1687-0247},
	journal = {Int. Math. Res. Not. IMRN},
	mrclass = {22E40},
	mrnumber = {4956566},
	number = {17},
	pages = {Paper No. rnaf268, 26},
	title = {Remarks on discrete subgroups with full limit sets in higher rank {L}ie groups},
	url = {https://doi.org/10.1093/imrn/rnaf268},
	year = {2025}}

@incollection{canary2008marden,
	author = {Canary, Richard D.},
	booktitle = {Geometry, analysis and topology of discrete groups},
	isbn = {978-1-57146-126-1; 978-7-04-023726-9},
	mrclass = {57M50 (57N10 57N45)},
	mrnumber = {2464394},
	mrreviewer = {Athanase\ Papadopoulos},
	pages = {137--162},
	publisher = {Int. Press, Somerville, MA},
	series = {Adv. Lect. Math. (ALM)},
	title = {Marden's tameness conjecture: history and applications},
	volume = {6},
	year = {2008}}

@article{kahn-wright,
	author = {Kahn, Jeremy and Wright, Alex},
	doi = {10.1215/00127094-2020-0049},
	fjournal = {Duke Mathematical Journal},
	issn = {0012-7094,1547-7398},
	journal = {Duke Math. J.},
	mrclass = {57M07 (30C62 30F40)},
	mrnumber = {4255043},
	mrreviewer = {Bruno\ P.\ Zimmermann},
	number = {3},
	pages = {503--573},
	title = {Nearly {F}uchsian surface subgroups of finite covolume {K}leinian groups},
	url = {https://doi.org/10.1215/00127094-2020-0049},
	volume = {170},
	year = {2021}}

@article{kahn-markovic,
	author = {Kahn, Jeremy and Markovic, Vladimir},
	doi = {10.4007/annals.2012.175.3.4},
	fjournal = {Annals of Mathematics. Second Series},
	issn = {0003-486X,1939-8980},
	journal = {Ann. of Math. (2)},
	mrclass = {57M50 (30F40)},
	mrnumber = {2912704},
	mrreviewer = {James\ W.\ Anderson},
	number = {3},
	pages = {1127--1190},
	title = {Immersing almost geodesic surfaces in a closed hyperbolic three manifold},
	url = {https://doi.org/10.4007/annals.2012.175.3.4},
	volume = {175},
	year = {2012}}

@article{fisher2024commensurators,
	author = {Fisher, David and Mj, Mahan and Van Limbeek, Wouter},
	journal = {Journal de l'{\'E}cole polytechnique---Math{\'e}matiques},
	pages = {1099--1122},
	publisher = {{\'E}cole polytechnique; CNRS; Universit{\'e} Joseph Fourier},
	title = {Commensurators of normal subgroups of lattices},
	volume = {11},
	year = {2024}}

@book{Zimmer,
	author = {Zimmer, Robert J.},
	doi = {10.1007/978-1-4684-9488-4},
	isbn = {3-7643-3184-4},
	mrclass = {22E40 (22D40 28D15)},
	mrnumber = {776417},
	mrreviewer = {S.\ G.\ Dani},
	pages = {x+209},
	publisher = {Birkh\"auser Verlag, Basel},
	series = {Monographs in Mathematics},
	title = {Ergodic theory and semisimple groups},
	url = {https://doi.org/10.1007/978-1-4684-9488-4},
	volume = {81},
	year = {1984}}

@incollection{gromov-hyperbolic,
	author = {Gromov, M.},
	booktitle = {Essays in group theory},
	doi = {10.1007/978-1-4613-9586-7\_3},
	isbn = {0-387-96618-8},
	mrclass = {20F32 (20F06 20F10 22E40 53C20 57R75 58F17)},
	mrnumber = {919829},
	mrreviewer = {Christopher\ W.\ Stark},
	pages = {75--263},
	publisher = {Springer, New York},
	series = {Math. Sci. Res. Inst. Publ.},
	title = {Hyperbolic groups},
	url = {https://doi.org/10.1007/978-1-4613-9586-7_3},
	volume = {8},
	year = {1987}}

@article{MR1934013,
	author = {Bridson, Martin R. and Howie, James and Miller, III, Charles F. and Short, Hamish},
	doi = {10.1023/A:1019611419598},
	fjournal = {Geometriae Dedicata},
	issn = {0046-5755,1572-9168},
	journal = {Geom. Dedicata},
	mrclass = {20E07 (20F34 20J05 57M05)},
	mrnumber = {1934013},
	mrreviewer = {Stephen\ J.\ Pride},
	note = {Dedicated to John Stallings on the occasion of his 65th birthday},
	pages = {95--103},
	title = {The subgroups of direct products of surface groups},
	url = {https://doi.org/10.1023/A:1019611419598},
	volume = {92},
	year = {2002}}

@article{chatterji2009discrete,
	author = {Chatterji, Indira and Venkataramana, T. N.},
	doi = {10.5802/jep.318},
	fjournal = {Journal de l'\'Ecole polytechnique. Math\'ematiques},
	issn = {2429-7100,2270-518X},
	journal = {J. \'Ec. polytech. Math.},
	mrclass = {22E40 (20G30)},
	mrnumber = {4972902},
	pages = {1605--1632},
	title = {Discrete linear groups containing arithmetic groups},
	url = {https://doi.org/10.5802/jep.318},
	volume = {12},
	year = {2025}}

@article{ELC-I,
	author = {Minsky, Yair},
	doi = {10.4007/annals.2010.171.1},
	fjournal = {Annals of Mathematics. Second Series},
	issn = {0003-486X,1939-8980},
	journal = {Ann. of Math. (2)},
	mrclass = {30F40 (20H10 57M50)},
	mrnumber = {2630036},
	mrreviewer = {Athanase\ Papadopoulos},
	number = {1},
	pages = {1--107},
	title = {The classification of {K}leinian surface groups. {I}. {M}odels and bounds},
	url = {https://doi.org/10.4007/annals.2010.171.1},
	volume = {171},
	year = {2010}}

@article{ELC-II,
	author = {Brock, Jeffrey F. and Canary, Richard D. and Minsky, Yair N.},
	doi = {10.4007/annals.2012.176.1.1},
	fjournal = {Annals of Mathematics. Second Series},
	issn = {0003-486X,1939-8980},
	journal = {Ann. of Math. (2)},
	mrclass = {57M50 (30F40)},
	mrnumber = {2925381},
	mrreviewer = {Athanase\ Papadopoulos},
	number = {1},
	pages = {1--149},
	title = {The classification of {K}leinian surface groups, {II}: {T}he ending lamination conjecture},
	url = {https://doi.org/10.4007/annals.2012.176.1.1},
	volume = {176},
	year = {2012}}

@misc{MoisheKohan,
    author = {Kohan, Moishe},
	howpublished = {\url{https://mathoverflow.net/questions/483528/ergodicity-of-action-of-finite-index-subgroups-in-the-boundary}},
	note = {MathOverflow, Accessed: 2025-12-07},
	title = {Ergodicity of action of finite index subgroups in the boundary}}

@incollection{selberg-lemma,
	author = {Selberg, Atle},
	booktitle = {Contributions to function theory ({I}nternat. {C}olloq. {F}unction {T}heory, {B}ombay, 1960)},
	mrclass = {20.65 (22.70)},
	mrnumber = {130324},
	mrreviewer = {E.\ Grosswald},
	pages = {147--164},
	publisher = {Tata Inst. Fund. Res., Bombay},
	title = {On discontinuous groups in higher-dimensional symmetric spaces},
	year = {1960}}

@article{mostow1980remarkable,
	author = {Mostow, George},
	journal = {Pacific Journal of Mathematics},
	number = {1},
	pages = {171--276},
	publisher = {Mathematical Sciences Publishers},
	title = {On a remarkable class of polyhedra in complex hyperbolic space},
	volume = {86},
	year = {1980}}

@article{Vinberg1974Field,
	author = {Vinberg, Ernest B.},
	journal = {Russian Mathematical Surveys},
	note = {Translated from Uspekhi Mat. Nauk 29:2(176) (1974), 209--210},
	number = {2},
	pages = {217--218},
	title = {The field of definition of a linear group},
	volume = {29},
	year = {1974}}

@article{Venkataramana0,
	author = {Venkataramana, T. N.},
	doi = {10.1016/0021-8693(87)90106-2},
	fjournal = {Journal of Algebra},
	issn = {0021-8693},
	journal = {J. Algebra},
	mrclass = {20G30 (20H05 22E40)},
	mrnumber = {892908},
	mrreviewer = {Alexander\ Lubotzky},
	number = {2},
	pages = {325--339},
	title = {Zariski dense subgroups of arithmetic groups},
	url = {https://doi.org/10.1016/0021-8693(87)90106-2},
	volume = {108},
	year = {1987}}

@article{Venkataramana1,
	author = {Venkataramana, T. N.},
	doi = {10.1007/BF02808202},
	fjournal = {Israel Journal of Mathematics},
	issn = {0021-2172,1565-8511},
	journal = {Israel J. Math.},
	mrclass = {22E40 (22E50)},
	mrnumber = {1324463},
	mrreviewer = {Alexander\ Starkov},
	number = {1-3},
	pages = {227--236},
	title = {On some rigid subgroups of semisimple {L}ie groups},
	url = {https://doi.org/10.1007/BF02808202},
	volume = {89},
	year = {1995}}

@article{lyons-sullivan,
	author = {Lyons, Terry and Sullivan, Dennis},
	fjournal = {Journal of Differential Geometry},
	issn = {0022-040X,1945-743X},
	journal = {J. Differential Geom.},
	mrclass = {58G32 (31C12 60J65)},
	mrnumber = {755228},
	mrreviewer = {P.\ E.\ Jupp},
	number = {2},
	pages = {299--323},
	title = {Function theory, random paths and covering spaces},
	url = {http://projecteuclid.org/euclid.jdg/1214438681},
	volume = {19},
	year = {1984}}

@article{uludag,
	author = {Uluda\u g, A. Muhammed},
	fjournal = {Osaka Journal of Mathematics},
	issn = {0030-6126},
	journal = {Osaka J. Math.},
	mrclass = {30F20 (30F15 30F25 30F35 53C20)},
	mrnumber = {1833622},
	mrreviewer = {Shigeo\ Segawa},
	number = {2},
	pages = {295--301},
	title = {Existence of {G}reen function and bounded harmonic functions on {G}alois covers of {R}iemannian manifolds},
	url = {http://projecteuclid.org/euclid.ojm/1153492466},
	volume = {38},
	year = {2001}}

@misc{stacks-project,
	author = {The {Stacks project authors}},
	howpublished = {\url{https://stacks.math.columbia.edu}},
	title = {The Stacks project},
	year = {2025}}
\bibliographystyle{alpha}

\end{document}